\documentclass[11pt,fleqn]{article}
\usepackage[obeyspaces,hyphens,spaces]{url}
\usepackage[dvipsnames]{xcolor}
\usepackage{etoolbox}
\usepackage{amsmath}
\usepackage{amssymb}
\usepackage{euscript}
\usepackage{mathrsfs} 
\usepackage{pifont}  
\usepackage{mathtools}
\usepackage{srcltx}  
\usepackage{charter}
\allowdisplaybreaks 
\definecolor{labelkey}{HTML}{0455BF}
\definecolor{refkey}{rgb}{0,0.6,0.0}
\definecolor{dblue}{HTML}{0455BF}
\definecolor{dgreen}{HTML}{02724A}
\definecolor{myellow}{HTML}{D97904}
\definecolor{dred}{HTML}{D90404}
\usepackage{upref}
\usepackage{hyperref}
\hypersetup{colorlinks=true,linktocpage=true,linkcolor=dblue,%
citecolor=dgreen,urlcolor=dred}

\newcommand{\Ccart}{\ensuremath{\raisebox{-0.5mm}%
{\mbox{\Large{$\times$}}}}}
\newcommand{\Cart}{\ensuremath{\raisebox{-0.5mm}%
{\mbox{\LARGE{$\times$}}}}}
\renewcommand{\leq}{\ensuremath{\leqslant}}
\renewcommand{\geq}{\ensuremath{\geqslant}}

\newcommand{\Argmin}{\ensuremath{\text{\rm Argmin}\,}}

\newcommand{\minimize}[2]{\ensuremath{\underset{\substack{{#1}}}%
{\text{\rm minimize}}\;\;#2}}

\newcommand{\scal}[2]{{\langle{{#1}\mid{#2}}\rangle}}
\newcommand{\pair}[2]{\langle{{#1},{#2}}\rangle}

\newcommand{\menge}[2]{\big\{{#1}~|~{#2}\big\}} 
 
\newcommand{\XXX}{\ensuremath{\boldsymbol{\mathcal{X}}}}
\newcommand{\ZZZ}{\ensuremath{\boldsymbol{\mathcal{Z}}}}
\newcommand{\UU}{\ensuremath{{\mathcal{U}}}}

\newcommand{\XX}{\ensuremath{{\mathcal{X}}}}
\newcommand{\YY}{\ensuremath{{\mathcal{Y}}}}

\newcommand{\NN}{\ensuremath{\mathbb{N}}}

\newcommand{\Sum}{\ensuremath{\displaystyle\sum}}

\newcommand{\emp}{\ensuremath{\varnothing}}

\newcommand{\Id}{\ensuremath{\mathrm{Id}}}

\newcommand{\RR}{\ensuremath{\mathbb{R}}}
\newcommand{\RP}{\ensuremath{\left[0,{+}\infty\right[}}
\newcommand{\RM}{\ensuremath{\left]{-}\infty,0\right]}}
\newcommand{\RMM}{\ensuremath{\left]{-}\infty,0\right[}}
\newcommand{\mm}{\ensuremath{\varpi}}

\newcommand{\RPP}{\ensuremath{\left]0,{+}\infty\right[}}

\newcommand{\RPX}{\ensuremath{\left[0,{+}\infty\right]}}
\newcommand{\RX}{\ensuremath{\left]{-}\infty,{+}\infty\right]}}
\newcommand{\RXX}{\ensuremath{\left[{-}\infty,{+}\infty\right]}}

\newcommand{\weakly}{\ensuremath{\rightharpoonup}}
\newcommand{\exi}{\ensuremath{\exists\,}}

\newcommand{\zer}{\text{\rm zer}\,}
\newcommand{\pinf}{\ensuremath{{{+}\infty}}}
\newcommand{\minf}{\ensuremath{{{-}\infty}}}

\newcommand{\dom}{\ensuremath{\text{\rm dom}\,}}
\newcommand{\barc}{\ensuremath{\text{\rm bar}\,}}
\newcommand{\rec}{\ensuremath{\text{\rm rec}\,}}
\newcommand{\cdom}{\ensuremath{\overline{\text{\rm dom}}\,}}
\newcommand{\prox}{\ensuremath{\text{\rm prox}}}
\newcommand{\proj}{\ensuremath{\text{\rm proj}}}

\newcommand{\gra}{\ensuremath{\text{\rm gra}\,}}
\newcommand{\inte}{\ensuremath{\text{\rm int}\,}}

\newcommand{\sri}{\ensuremath{\text{\rm sri}\,}}

\newcommand{\infconv}{\ensuremath{\mbox{\small$\,\square\,$}}}

\newcommand{\rocky}{\ensuremath{\mbox{\footnotesize$\odot$}}}

\def\abstract{\noindent{\bfseries Abstract}. \ignorespaces}

\usepackage{theorem}
\newtheorem{theorem}{Theorem}[section]
\newtheorem{lemma}[theorem]{Lemma}

\newtheorem{proposition}[theorem]{Proposition}

\theoremstyle{plain}{\theorembodyfont{\rmfamily}%
}

\theoremstyle{plain}{\theorembodyfont{\rmfamily}%
\newtheorem{example}[theorem]{Example}}
\theoremstyle{plain}{\theorembodyfont{\rmfamily}%
\newtheorem{remark}[theorem]{Remark}}
\theoremstyle{plain}{\theorembodyfont{\rmfamily}%
}
\theoremstyle{plain}{\theorembodyfont{\rmfamily}%
}
\theoremstyle{plain}{\theorembodyfont{\rmfamily}%
\newtheorem{definition}[theorem]{Definition}}
\theoremstyle{plain}{\theorembodyfont{\rmfamily}%
\newtheorem{problem}[theorem]{Problem}}
\theoremstyle{plain}{\theorembodyfont{\rmfamily}%
\newtheorem{notation}[theorem]{Notation}}
\usepackage{enumitem}
\setlist[enumerate]{itemsep=2pt}
\setlist[itemize]{itemsep=2pt}

\numberwithin{equation}{section}
\usepackage{authblk}

\newcommand{\email}[1]{\href{mailto:#1}{\nolinkurl{#1}}}
\author[1]{Luis M. Brice\~{n}o-Arias}
\author[2]{Patrick L. Combettes}
\affil[1]{Universidad T\'ecnica Federico Santa Mar\'ia,
Departamento de Matem\'atica, Santiago, Chile
\email{luis.briceno@usm.cl}\medskip}
\affil[2]{North Carolina State University,
Department of Mathematics, Raleigh, USA

\email{plc@math.ncsu.edu}
}

\begin{document}

\title{\sffamily\huge A Perturbation Framework for Convex
Minimization and Monotone Inclusion Problems with Nonlinear
Compositions\thanks{Contact author: 
P. L. Combettes.
Email: \email{plc@math.ncsu.edu}.
Phone: +1 919 515 2671.
The work of L. M. Brice\~{n}o-Arias was supported by Centro de
Modelamiento Matem\'atico (CMM), ACE210010 and FB210005, BASAL
funds for centers of excellence and FONDECYT 1190871 from
ANID-Chile, and the work of P. L. Combettes was supported by the
National Science Foundation under grant DMS-1818946.}
}

\date{~}

\maketitle

\begin{abstract} 
\noindent 
We introduce a framework based on Rockafellar's perturbation theory
to analyze and solve general nonsmooth convex minimization and
monotone inclusion problems involving nonlinearly composed
functions as well as linear compositions. Such problems have been
investigated only from a primal perspective and only for nonlinear
compositions of smooth functions in finite-dimensional spaces in
the absence of linear compositions. In the context of Banach
spaces, the proposed perturbation analysis serves as a foundation
for the construction of a dual problem and of a maximally monotone
Kuhn--Tucker operator which is decomposable as the sum of simpler
monotone operators. In the Hilbertian setting, this decomposition
leads to a block-iterative primal-dual algorithm that fully
splits all the components of the problem and appears to be the 
first proximal splitting algorithm for handling nonlinear composite
problems. Various applications are discussed.
\end{abstract}


\noindent
{\bfseries MSC2010 subject classification:}
46N10, 47J20, 49M27, 49M29, 65K05, 90C25.

\section{Introduction}
This paper concerns the analysis and the numerical solution of
optimization and monotone inclusion problems involving nonlinear
compositions of convex functions. The novelty and difficulty of
such formulations reside in the nonlinear composite terms and, in
particular, in the design of a splitting mechanism that will fully
decompose them towards a fine convex analysis of the problem and
the development of efficient numerical methods. To isolate this
difficulty, we first study the following simpler minimization
formulation, where $\Gamma_0(\XX)$ designates the class of proper
lower semicontinuous convex functions from a Banach space $\XX$ to
$\RX$, $\sri$ the strong relative interior, and $\infconv$ the
infimal convolution operation (see Section~\ref{sec:2} for
notation).

\begin{problem}
\label{prob:1}
Let $\XX$ and $\YY$ be reflexive real Banach spaces, let
$\phi\in\Gamma_0(\RR)$ be increasing and not constant, let
$f\in\Gamma_0(\XX)$, let $g\in\Gamma_0(\YY)$, let
$\ell\in\Gamma_0(\YY)$, let $L\colon\XX\to\YY$ be linear and 
bounded, and let $h\in\Gamma_0(\XX)$. Set
\begin{equation}
\label{e:c1}
\phi\circ f\colon\XX\to\RX\colon x\mapsto
\begin{cases}
\phi\big(f(x)\big),&\text{if}\;\;f(x)\in\dom\phi;\\
\pinf,&\text{if}\;\;f(x)\notin\dom\phi
\end{cases}
\end{equation}
and suppose that the following hold: 
\begin{enumerate}[label={\rm[\alph*]}]
\item 
\label{a:1i}
$(\exi z\in\dom f)$ $f(z)\in\inte\dom\phi$.
\item 
\label{a:1iii}
$0\in\sri(f^{-1}(\dom\phi)-\dom h)$.
\item 
\label{a:1ii}
$0\in\sri(\dom g^*-\dom \ell^*)$.
\item 
\label{a:1iv}
$0\in\sri(L(\dom h\cap f^{-1}(\dom\phi))-\dom g-\dom\ell)$.
\end{enumerate}
The goal is to 
\begin{equation}
\label{e:p1}
\minimize{x\in\XX}{(\phi\circ f)(x)+(g\infconv\ell)(Lx)+h(x)},
\end{equation}
and the set of solutions is denoted by $\mathscr{P}$.
\end{problem}

To motivate our investigation, let us consider a few notable
special cases of Problem~\ref{prob:1}.

\begin{example}
\label{ex:1}
In Problem~\ref{prob:1} suppose that $\phi=\iota_{\RM}$. Then
\eqref{e:p1} reduces to the constrained minimization problem
\begin{equation}
\label{e:pineq}
\minimize{\substack{x\in\XX\\ f(x)\leq 0}}
{(g\infconv\ell)(Lx)+h(x)},
\end{equation}
which is pervasive in nonlinear programming. 
For instance, suppose that $\XX=\RR^N$ and that
$f=\|\cdot\|^p_p-\eta^p$, where 
$\eta\in\RPP$ and 
$p\in\left[1,\pinf\right[$. Then \eqref{e:pineq} becomes
\begin{equation}
\label{e:pineq2}
\minimize{\substack{x\in\RR^N\\ \|x\|_p\leq\eta}}
{(g\infconv\ell)(Lx)+h(x)}.
\end{equation}
An instance of \eqref{e:pineq2} in the context of machine learning
is found in \cite[Section~4.1]{Jagg13}. 
\end{example} 

\begin{example}
\label{ex:3}
Let $\theta\in\Gamma_0(\RR)$ be an increasing function such that
$\dom\theta=\left]\minf,\eta\right[$ for some $\eta\in\RPX$,
$\lim_{\xi\uparrow\eta}\theta(\xi)=\pinf$, and 
$(\rec\theta)(1)>0$. Let $\alpha\colon\RPP\to\RPP$ be such that 
$\lim_{\rho\downarrow 0}\alpha(\rho)=0$ and 
$\varliminf_{\rho\downarrow 0}{\alpha(\rho)}/{\rho}>0$.
In Problem~\ref{prob:1}, set $\XX=\YY=\RR^N$, $h=0$, 
$\ell=\iota_{\{0\}}$, $L=\Id$, and given $\rho\in\RPP$, 
$\phi\colon\xi\mapsto\alpha(\rho)\theta(\xi/\rho)$. 
Then \eqref{e:p1} becomes
\begin{equation}
\minimize{x\in\RR^N}{\alpha(\rho)\theta\big(f(x)/\rho\big)+g(x)}.
\end{equation}
The asymptotic behavior of this family of penalty-barrier
minimization problems as $\rho\downarrow 0$ is investigated in 
\cite{Ausl97}.
\end{example}

\begin{example}
\label{ex:2}
In Problem~\ref{prob:1} suppose that $\ell=\iota_{\{0\}}$ and
$\phi=\theta\circ\max\{0,\cdot\}$, where $\theta\in\Gamma_0(\RR)$ 
is even with $\Argmin\theta=\{0\}$. Then
$\phi=\theta\circ\max\{0,\cdot\}$ is an increasing function in
$\Gamma_0(\RR)$ which is not constant and \eqref{e:p1} reduces to 
\begin{equation}
\minimize{x\in\XX}{\theta\big(\max\{0,f(x)\}\big)+g(Lx)+h(x)}.
\end{equation}
For instance, if $C$ is a nonempty closed convex subset of $\XX$
and $f=d_C-\varepsilon$, where $\varepsilon\in\RP$, we recover a
scenario discussed in \cite{Siim19}. A special case is given by
$f=d_C$ and $\theta\colon\xi\mapsto\ln(\rho)-\ln(\rho-\xi)$
if $\xi<\rho$; and $\pinf$ if $\xi\geq\rho$, where $\rho\in\RPP$.
Thus, $\phi\circ f=\ln(\rho)-\ln(\rho-d_C)$ if $d_C<\rho$; and 
$\pinf$ if $d_C\geq\rho$, and it acts as a barrier keeping
solutions at distance at most $\rho$ from the set $C$.
\end{example}

\begin{example}
\label{ex:6}
Let $p\in\left[1,\pinf\right[$ and let 
$(\XX_i)_{i\in I}$ and $(\YY_k)_{k\in K}$ be finite families of
reflexive real Banach spaces. 
For every $i\in I$, let $C_i$ be a nonempty closed convex 
subset of $\XX_i$ and, for every $k\in K$, let 
$g_k\in\Gamma_0(\YY_k)$ and let 
$L_{k,i}\colon\XX_i\to\YY_k$ be a bounded linear operator. 
In Problem~\ref{prob:1}, set $\XX=\bigoplus_{i\in I}\XX_i$,
$\YY=\bigoplus_{k\in K}\YY_k$, 
$f\colon (x_i)_{i\in I}\mapsto(\sum_{i\in I}d_{C_i}^p(x_i))^{1/p}$,
$g\colon (y_k)_{k\in K}\mapsto\sum_{k\in K}g_k(y_k)$, 
$\ell=\iota_{\{0\}}$,
$L\colon(x_i)_{i\in I}\mapsto(\sum_{i\in I}L_{k,i}x_i)_{k\in K}$, 
and $\phi=(\max\{0,\cdot\})^p$. Then \eqref{e:p1} reduces to 
\begin{equation}
\minimize{(x_i)_{i\in I}\in\underset{i\in I}{\Ccart}\XX_i}
{\sum_{i\in I}d^p_{C_i}(x_i)+\sum_{k\in K}g_k
\bigg(\sum_{i\in I}L_{k,i}x_i\bigg)+h\big((x_i)_{i\in I}\big)}.
\end{equation}
This formulation covers signal processing and
location problems \cite{Nmtm09,Siim19,Mord12}.
\end{example}

There is a vast literature on Problem~\ref{prob:1} in the case of
linear compositions, that is, when $\phi\colon\xi\mapsto\xi$. In
this context, the duality theory goes back to \cite{Rock67} when
$h=0$ and $\ell=\iota_{\{0\}}$. On the algorithmic front, the
primal-dual strategy adopted in \cite{Svva12} when $h$ is smooth
and $\ell$ is strongly convex, which originates in \cite{Siop11},
is ultimately rooted in Fenchel--Rockafellar duality and boils down
to finding a zero of an associated Kuhn--Tucker operator via
monotone operator splitting methods. Further algorithmic
developments along these lines in the linear composition setting
can be found in \cite{Siop11,Chen94,Svva12,Cond13,Bang13}.
Unfortunately this methodology is not extendible to
Problem~\ref{prob:1} in the presence of a nonlinear function
$\phi$. In the nonlinear setting, in terms of convex analysis, the
conjugate and the subdifferential of $\phi\circ f$ in \eqref{e:c1}
are derived in \cite{Thib94,Thib96}, building up on the work
of \cite{Kuta77,Lema85,Lesc68,Rubi77}; see also \cite{Burk21} for
the Euclidean setting. Optimality conditions for
the minimization of $\phi\circ f+h$ are established in
\cite{Penn99,Penn00}. Duality theory for the minimization
problem \eqref{e:c1} does not seem to have been studied, even in
the case when $h=0$ and $\ell=\iota_{\{0\}}$. On the numerical
side, in the finite-dimensional setting, with $f$ smooth,
$\YY=\XX$, and $g=h=\ell^*=0$, Problem~\ref{prob:1} has been
studied in \cite{Bolt20} (in the case when $f$ is vector-valued) 
by linearizing the objective function; see also
\cite{Burk85,Burk20,Lewi16,Pauw16,Repe21} and the references
therein for alternative approximations of $f$ in this type of
scenario. However, in the general setting of
Problem~\ref{prob:1}, solution methods are not available. 

To address the gaps identified above, we propose a methodology
based on Rockafellar's perturbation theory. This general theory 
was initiated in \cite{Rock69} and further developed in
\cite{Roc70a,Rock74}; see also \cite{Joly71}. Specifically, we
introduce for Problem~\ref{prob:1} the perturbation function
\begin{align}
\label{e:F}
F\colon\XX\times\RR\times\YY&\to\RX\nonumber\\
(x,\xi,y)&\mapsto 
\begin{cases}
\phi\big(f(x)+\xi\big)+(g\infconv\ell)(Lx+y)+h(x),&\text{if}\;\; 
f(x)+\xi\in\dom\phi;\\
\pinf,&\text{if}\;\;f(x)+\xi\notin\dom\phi.
\end{cases}
\end{align}
In this context, perturbation variables $\xi\in\RR$ and $y\in\YY$
are associated to the nonlinear and the linear composition,
respectively, and Problem~\ref{prob:1} can be expressed as 
\begin{equation}
\minimize{x\in\XX}{F(x,0,0)}.
\end{equation}
By specializing the general theory of \cite{Rock74} to the
perturbation function \eqref{e:F}, we derive a dual for
Problem~\ref{prob:1} involving variables for each composition,
together with a Kuhn--Tucker operator. We then show that this
Kuhn--Tucker operator can be decomposed as a sum of elementary
maximally monotone operators. In the Hilbertian setting, we derive
their resolvents and propose proximal splitting algorithms that use
$\phi$, $f$, $g$, $\ell$, $h$, and $L$ separately to solve
Problem~\ref{prob:1} and its dual. In this endeavor, we leverage
the fact that, while the proximity operator of $\phi\circ f$ is
usually unknown, those of $\phi$ and $f$ are often available. The
resulting algorithms capture methods that were known
in the case when $\phi\colon\xi\mapsto\xi$.

Next, we consider the broader setting of monotone inclusions. We
first note that the minimization Problem~\ref{prob:1} is a special
case of the inclusion problem 
\begin{equation}
\label{e:p31}
\text{find}\;\;x\in\XX\;\;\text{such that}\;\; 
0\in\partial(\phi\circ f)(x)
+\big(L^*\circ(B\infconv D)\circ L\big)x+Ax,
\end{equation}
involving a potential term $\phi\circ f$ and the maximally 
monotone operators $A=\partial h$, $B=\partial g$, and
$D=\partial\ell$, where $B\infconv D$ designates the
parallel sum of $B$ and $D$. For general maximally monotone 
operators this type of inclusion problem mixing
potential and nonpotential terms appears in several areas,
including game theory, saddle problems, evolution equations,
variational inequalities, deep neural networks, and
diffusion/convection problems in physics; see, e.g.,
\cite{Adly22,Alwa21,Atto10,Siop11,Bric13,Svva20,Sign21,%
Facc03,Lion69,Zeid90}. Building up on the tools developed for
Problem~\ref{prob:1}, we shall address the following general form
of \eqref{e:p31}.

\begin{problem}
\label{prob:11}
Let $\XX$ be a reflexive real Banach space, let $I$ and $K$ be
disjoint finite sets, and let $(\beta,\chi)\in\RPP^2$. For every
$i\in I$, let $\phi_i\in\Gamma_0(\RR)$ be increasing and not 
constant and let $f_i\in\Gamma_0(\XX)$ be such that
$(\inte\dom\phi_i)\cap f_i(\dom f_i)\neq\emp$. 
For every $k\in K$, let $\YY_k$ be a reflexive real
Banach space, let $B_k\colon\YY_k\to 2^{\YY_k^*}$ and
$D_k\colon\YY_k\to 2^{\YY_k^*}$ be maximally monotone, and let
$L_k\colon\XX\to\YY_k$ be linear and bounded. Furthermore, let
$A\colon\XX\to 2^{\XX^*}$ be maximally monotone, let
$C\colon\XX\to\XX^*$ be $\beta$-cocoercive, and let
$Q\colon\XX\to\XX^*$ be monotone and $\chi$-Lipschitzian. The goal
is to
\begin{equation}
\label{e:prob11}
\text{find}\;\; x\in\XX\;\;\text{such that}\;\; 
0\in\sum_{i\in I}\partial(\phi_i\circ f_i)(x)
+\sum_{k\in K}\big(L^*_k\circ(B_k\infconv D_k)\circ 
L_k\big)x+Ax+Cx+Qx.
\end{equation}
\end{problem}

An algorithmic framework will be proposed to solve
Problem~\ref{prob:11} in the Hilbertian setting and various 
applications will be discussed.
In particular, we shall design a proximal algorithm to solve
minimization problems involving sums of maxima of convex functions.

We present our notation and provide preliminary results in
Section~\ref{sec:2}. In Section~\ref{sec:3}, the perturbation $F$
introduced in \eqref{e:F} is employed to derive a dual for
Problem~\ref{prob:1} and its Kuhn--Tucker operator, which is shown
to be maximally monotone. We then obtain primal-dual solutions as
the zeros of this Kuhn--Tucker operator, which we decompose as the
sum of elementary monotone operators. This decomposition is
exploited in Section~\ref{sec:4} to derive splitting algorithms
that use $\phi$, $f$, $g$, $\ell$, $h$, and $L$ separately to solve
Problem~\ref{prob:1} when all the Banach spaces are Hilbertian.
Section~\ref{sec:5} is devoted to the analysis and the numerical
solution of Problem~\ref{prob:11} based on the tools developed in
Sections~\ref{sec:3} and~\ref{sec:4}. A highlight of that section
is a block-iterative algorithm that fully splits all the components
of Problem~\ref{prob:11}. Even in its nonblock-iterative
implementation, this appears to be the first proximal splitting
algorithm for handling nonlinear composite problems. Applications 
of the proposed framework are presented in Section~\ref{sec:6}.

\section{Notation and preliminary results}
\label{sec:2}
Let $(\XX,\|\cdot\|)$ and $(\YY,\|\cdot\|)$ be reflexive real
Banach spaces, let $\XX^*$ and $\YY^*$ denote their respective
topological duals, and let $\scal{\cdot}{\cdot}$ denote the
standard bilinear forms on $\XX\times\XX^*$ and $\YY\times\YY^*$.
Throughout, we use the following notational conventions.

\begin{notation}
\label{n:1}
A generic vector in $\XX$ is denoted by $x$ and
a generic vector in $\XX^*$ is denoted by $x^*$.
Bold symbols such as $\boldsymbol{z}$ denote elements in product
spaces.
\end{notation}

The symbol $\XX\oplus\YY$ designates the standard product 
vector space $\XX\times\YY$ equipped with the pairing
\begin{equation}
\big(\forall(x,y)\in\XX\times\YY\big)
\big(\forall(x^*,y^*)\in\XX^*\times\YY^*\big)\quad
\pair{(x,y)}{(x^*,y^*)}=\pair{x}{x^*}+\pair{y}{y^*}
\end{equation}
and the norm
\begin{equation}
\label{e:n}
\big(\forall(x,y)\in\XX\times\YY\big)
\quad\|(x,y)\|=\sqrt{\|x\|^2+\|y\|^2}.
\end{equation}
The power set of $\XX^*$ is denoted by $2^{\XX^*}$. Let 
$A\colon\XX\to 2^{\XX^*}$ be a set-valued operator. We denote by
$\dom A=\menge{x\in\XX}{Ax\neq\emp}$ the domain of $A$,
by $\zer A=\menge{x\in\XX}{0\in Ax}$ the set of zeros of $A$, by 
$\gra A=\menge{(x,x^*)\in\XX\times\XX^*}{x^*\in Ax}$ the graph of
$A$, and by $A^{-1}$ the inverse of $A$, which has graph 
$\menge{(x^*,x)\in\XX^*\times\XX}{x^*\in Ax}$. The parallel sum of
$A$ and $B\colon\XX\to 2^{\XX^*}$ is
\begin{equation}
\label{e:94}
A\infconv B=\big(A^{-1}+B^{-1}\big)^{-1}.
\end{equation}
Moreover, $A$ is monotone if
\begin{equation}
\big(\forall (x,x^*)\in\gra A\big)
\big(\forall (y,y^*)\in\gra A\big)
\quad\pair{x-y}{x^*-y^*}\geq 0,
\end{equation}
and maximally so if there exists no monotone operator
$B\colon\XX\to 2^{\XX^*}$ such that 
$\gra A\subset\gra B\neq\gra A$. If $\XX$ is Hilbertian and $A$ is
maximally monotone, $J_A=(\Id+A)^{-1}$ is the resolvent of $A$,
which is single-valued with $\dom J_A=\XX$. If $A$ is
single-valued and $\beta\in\RPP$, $A$ is $\beta$-cocoercive if 
\begin{equation}
(\forall x\in\XX)(\forall y\in\XX)\quad
\pair{x-y}{Ax-Ay}\geq\beta\|Ax-Ay\|^2.
\end{equation}
Let $f\colon\XX\to\RXX$. Then $f$ is proper if 
$\minf\notin f(\XX)\neq\{\pinf\}$, the domain of $f$
is $\dom f=\menge{x\in\XX}{f(x)<\pinf}$, and the set
of minimizers of $f$ is 
$\Argmin f=\menge{x\in\dom f}{(\forall y\in\XX)\;f(x)\leq f(y)}$.
The conjugate of $f$ is 
$f^*\colon x^*\mapsto\sup_{x\in\XX}(\pair{x}{x^*}-f(x))$.
If $f$ is proper, its subdifferential is 
\begin{equation}
\label{e:subdiff}
\partial f\colon\XX\to 2^{\XX^*}\colon x\mapsto 
\menge{x^*\in\XX^*}{(\forall y\in\dom f)\;\:
\pair{y-x}{x^*}+f(x)\leq f(y)},
\end{equation}
its recession function is 
$\rec f\colon\XX\to\RX\colon y\mapsto\sup_{x\in\dom f}
(f(x+y)-f(x))$, and its perspective is
\begin{equation}
\label{e:pers}
\widetilde{f}\colon\XX\times\RR\to\RX\colon (x,\xi)\mapsto
\begin{cases}
\xi\,f({x}/\xi),&\text{if}\;\:\xi>0;\\
(\rec{f})({x}),&\text{if}\;\:\xi=0;\\
\pinf,&\text{otherwise.}
\end{cases}
\end{equation}
We set
\begin{equation}
\label{e:66}
(\forall\xi\in\RP)\quad\xi\rocky f=
\begin{cases}
\iota_{\cdom f},&\text{if}\;\:\xi=0;\\
\xi f,&\text{if}\;\:\xi>0.
\end{cases}
\end{equation} 
The infimal convolution of $f$ with a proper function 
$\ell\colon\XX\to\RX$ is 
\begin{equation}
f\infconv\ell\colon\XX\to\RXX
\colon x\mapsto\inf_{y\in\XX}\big(f(y)+\ell(x-y)\big).
\end{equation}
Let $C$ be a convex subset of $\XX$. The indicator function of $C$
is denoted by $\iota_C$, the support function of $C$ by $\sigma_C$,
the strong relative interior of $C$, 
i.e., the set of points $x\in C$ such that the cone generated by
$-x+C$ is a closed vector subspace of $\XX$, by $\sri C$, and the
distance to $C$ by $d_C$, i.e., 
$d_C\colon x\mapsto\inf_{y\in C}\|x-y\|$. The barrier cone of $C$
is $\barc{C}=\menge{x^*\in\XX^*}{\sup\,\pair{C}{x^*}<\pinf}$ and
the normal cone operator of $C$ is $N_C=\partial\iota_C$.
Now suppose that $\XX$ is Hilbertian. If $f\in\Gamma_0(\XX)$, for
every $x\in\XX$, $\prox_fx$ denotes the unique minimizer of
$f+\|\cdot-x\|^2/2$ and 
$\prox_f=(\Id+\partial f)^{-1}=J_{\partial f}$
is called the proximity operator of $f$ and
$\prox_{f}(\XX)\subset\dom\partial f\subset\dom f$.
If $C$ is nonempty and closed,
$\proj_C=\prox_{\iota_C}$ is the projection operator onto $C$. For
background on convex analysis and monotone operators, see
\cite{Livre1,Zali02}.

The following facts will be required subsequently.

\begin{lemma}
\label{l:66}
Let $\XX$ be a reflexive real Banach space, let
$f\in\Gamma_0(\XX)$, and let $\xi\in\RP$. Then the following hold:
\begin{enumerate}
\item
\label{l:66i}
$\xi\rocky f\in\Gamma_0(\XX)$.
\item
\label{l:66ii}
$[\widetilde{f}(\cdot,\xi)]^*=\xi\rocky f^*$ and
$(\xi\rocky f)^*=\widetilde{f^*}(\cdot,\xi)$. 
\item
\label{l:66iii}
We have
\begin{equation}
\partial(\xi\rocky f)=
\begin{cases}
N_{\cdom f},&\text{if}\;\;\xi=0;\\
\xi\partial f,&\text{if}\;\;\xi\in\RPP.
\end{cases}
\end{equation}
\item
\label{l:66iv}
Suppose that $\XX$ is Hilbertian. Then 
\begin{equation}
\label{e:12}
\prox_{\xi\rocky f}=
\begin{cases}
\proj_{\cdom f}, &\text{if}\;\;\xi=0;\\
\prox_{\xi f}, &\text{if}\;\;\xi\in\RPP.
\end{cases}
\end{equation}
\end{enumerate}
\end{lemma}
\begin{proof}
\ref{l:66i}--\ref{l:66ii}: See \cite[Theorem~3E]{Rock66}.

\ref{l:66iii}--\ref{l:66iv}: Clear from \ref{l:66i} and 
\eqref{e:66}.
\end{proof}

\begin{lemma}
\label{l:66b}
{\rm \cite[Theorem~3F]{Rock66}}
Let $\XX$ be a reflexive real Banach space, let
$f\in\Gamma_0(\XX)$, and set
$C=\menge{(x^*,\xi^*)\in\XX^*\oplus\RR}{f^*(x^*)+\xi^*\leq 0}$. 
Then the following hold:
\begin{enumerate}
\item
\label{l:66bi}
$\widetilde{f}=\sigma_C$.
\item
\label{l:66bii}
$\widetilde{f}\in\Gamma_0(\XX\oplus\RR)$.
\end{enumerate}
\end{lemma}

\begin{lemma}
\label{l:6}
Let $\XX$ be a reflexive real Banach space, let
$f\colon\XX\to\RX$ be proper, let $x\in\XX$, and let 
$x^*\in\XX^*$. Then the following hold:
\begin{enumerate}
\item
\label{l:6i}
$f(x)+f^*(x^*)\geq\pair{x}{x^*}$.
{\rm (\cite[Theorem~2.3.1(ii)]{Zali02})}
\item
\label{l:6ii}
$f(x)+f^*(x^*)=\pair{x}{x^*}$ $\Leftrightarrow$
$(x,x^*)\in\gra\partial f$.
{\rm (\cite[Theorem~2.4.2(iii)]{Zali02})}
\end{enumerate}
\end{lemma}

\begin{lemma}
\label{l:2}
Let $\XX$ be a reflexive real Banach space and let 
$f\in\Gamma_0(\XX)$. Then the following hold:
\begin{enumerate}
\item
\label{l:2i}
$f^*\in\Gamma_0(\XX^*)$ and $f=f^{**}$. 
{\rm (\cite[Theorem~2.3.3]{Zali02})}
\item
\label{l:2ii}
$(\partial f)^{-1}=\partial f^*$. 
{\rm (\cite[Theorem~2.4.4(iv)]{Zali02})}
\item
\label{l:2iii}
$\partial f$ is maximally monotone.
{\rm (\cite[Theorem~3.2.8]{Zali02})}
\item
\label{l:2iv}
$\Argmin f=\zer\partial f$.
{\rm (\cite[Theorem~2.5.7]{Zali02})}
\end{enumerate}
\end{lemma}

\begin{lemma}
\label{l:1}
Let $\XX$ be a reflexive real Banach space, let
$f\in\Gamma_0(\XX)$, and let $g\in\Gamma_0(\XX)$. Suppose that
$0\in\sri(\dom f-\dom g)$. Then the following hold:
\begin{enumerate}
\item
\label{l:1i} 
$(f+g)^*=f^*\infconv g^*\in\Gamma_0(\XX^*)$. 
{\rm (\cite[Theorem~1.1]{Atto86})}
\item
\label{l:1ii} 
$\partial (f+g)=\partial f+\partial g$. 
{\rm (\cite[Corollary~2.1]{Atto86})}
\end{enumerate}
\end{lemma}

\begin{lemma}
\label{l:4}
Let $\XX$ be a reflexive real Banach space, let 
$f\in\Gamma_0(\XX)$, and let $g\in\Gamma_0(\XX)$. 
Suppose that $0\in\sri(\dom f^*-\dom g^*)$. Then 
$(\partial f)\infconv(\partial g)=\partial(f\infconv g)$.
\end{lemma}
\begin{proof}
It follows from \eqref{e:94}, Lemma~\ref{l:2}, and Lemma~\ref{l:1}
applied to $f^*$ and $g^*$ 
that $(\partial f)\infconv(\partial g) 
=((\partial f)^{-1}+(\partial g)^{-1})^{-1}
=(\partial f^*+\partial g^*)^{-1}
=(\partial(f^*+g^*))^{-1}
=\partial(f^*+g^*)^*
=\partial(f^{**}\infconv g^{**})
=\partial(f\infconv g)$. 
\end{proof}

\begin{lemma}
\label{l:3}
Let $\phi\colon\RR\to\RX$ be an increasing proper convex function.
Then the following hold:
\begin{enumerate}
\item
\label{l:3i}
$\dom\phi$ is an interval and $\inf\dom\phi=\minf$.
\item
\label{l:3ii}
$\dom\phi^*\subset\RP$.
\item 
\label{l:3iii}
Suppose that $\phi$ is lower semicontinuous and not constant. Then
$\dom\phi^*\cap\RPP\neq\emp$.
\end{enumerate}
\end{lemma}
\begin{proof}
\ref{l:3i}: Since $\phi$ is convex, $\dom\phi$ is convex, and
hence an interval. Now take $\xi\in\dom\phi$ and
$\eta\in\left]\minf,\xi\right[$. Then
$\phi(\eta)\leq\phi(\xi)<\pinf$ and therefore $\eta\in\dom\phi$.
Consequently, $\inf\dom\phi=\minf$.

\ref{l:3ii}:
Since $\phi$ is increasing, it follows from \ref{l:3i} that
\begin{equation}
(\forall\xi^*\in\RMM)\quad\phi^*(\xi^*)
=\sup_{\xi\in\dom\phi}\big(\xi\xi^*-\phi(\xi)\big)=\pinf,
\end{equation}
which yields $\dom\phi^*\subset\RP$.

\ref{l:3iii}: If $\dom\phi^*=\{0\}$, then
$\phi^*=\iota_{\{0\}}+\phi^*(0)$ and hence, since
$\phi\in\Gamma_0(\RR)$,
$\phi=\phi^{**}\equiv-\phi^*(0)$ is constant. Therefore, the claim
follows from \ref{l:3ii}.
\end{proof}

The next result discusses convex analytical properties of nonlinear
compositions. It provides in particular a connection between the
conjugate of $\phi\circ f$ and the marginal of a function that
involves the perspective of $f$ (see \eqref{e:pers}). 

\begin{proposition}
\label{p:1}
Let $\XX$ be a reflexive real Banach space, let
$\phi\in\Gamma_0(\RR)$ be a nonconstant increasing function, and
let $f\in\Gamma_0(\XX)$ be such that 
$(\dom\phi)\cap f(\dom f)\neq\emp$. Then the following hold: 
\begin{enumerate}
\item 
\label{p:1i}
$\dom(\phi\circ f)=f^{-1}(\dom\phi)$.
\item 
\label{p:1ii}
$\phi\circ f\in\Gamma_0(\XX)$.
\item 
\label{p:1iv}
Suppose that there exists $z\in \XX$ such that 
$f(z)\in\inte\dom\phi$, and let $x^*\in\XX^*$. Then 
\begin{equation}
\label{e:thib94}
(\phi\circ f)^*(x^*)
=\min_{\xi^*\in\RR}\big(\phi^*(\xi^*)+
\widetilde{f^*}(x^*,\xi^*)\big).
\end{equation}
\item 
\label{p:1iii-}
Let $x\in\dom f$. Then 
\begin{equation}
\bigcup_{\xi^*\in\partial\phi(f(x))}\partial(\xi^*\rocky f)(x)
\subset\partial(\phi\circ f)(x). 
\end{equation}
\item 
\label{p:1iii}
Suppose that there exists $z\in\dom f$ such that 
$f(z)\in\inte\dom\phi$, and let $x\in\dom f$. Then
\begin{equation}
\partial(\phi\circ f)(x)= 
\bigcup_{\xi^*\in\partial\phi(f(x))}\partial(\xi^*\rocky f)(x).
\end{equation}
\item
\label{p:1v-}
Suppose that there exists $z\in\dom f$ such that 
$f(z)\in\inte\dom\phi$, let $x\in\dom f$, let $x^*\in\XX^*$, and
let $\xi^*\in\RR$. Then the following are equivalent:
\begin{enumerate}
\item
\label{p:1v-a}
$(\phi\circ f)(x)+\phi^*(\xi^*)=(\xi^*\rocky f)(x)$\;\; and\;\;
$(\xi^*\rocky f)(x)+(\xi^*\rocky f)^*(x^*)=\pair{x}{x^*}$.
\item
\label{p:1v-b}
$(\phi\circ f)(x)+\phi^*(\xi^*)+(\xi^*\rocky f)^*(x^*)=
\pair{x}{x^*}$.
\item
\label{p:1v-c}
$(\phi\circ f)(x)+(\phi\circ f)^*(x^*)=\pair{x}{x^*}$\;\;and\;\;
$\phi^*(\xi^*)+(\xi^*\rocky f)^*(x^*)=(\phi\circ f)^*(x^*)$.
\end{enumerate}
\item
\label{p:1v}
Suppose that there exists $z\in\dom f$ such that 
$f(z)\in\inte\dom\phi$, let $x\in\dom f$, let $x^*\in\XX^*$, and
let $\xi^*\in\RR$. Then 
\begin{equation}
\label{e:burk}
\begin{cases}
\xi^*\in\partial\phi\big(f(x)\big)\\
x^*\in\partial(\xi^*\rocky f)(x)
\end{cases}
\Leftrightarrow\quad 
\begin{cases}
x^*\in\partial (\phi\circ f)(x)\\
\xi^*\in\Argmin\big(\phi^*+\widetilde{f^*}(x^*,\cdot)\big).
\end{cases}
\end{equation}
\end{enumerate}
\end{proposition}
\begin{proof}
\ref{p:1i}: See \eqref{e:c1}.

\ref{p:1ii}: 
In view of \eqref{e:c1},
convexity is established as in \cite[Proposition~8.21]{Livre1}. 
Since $\phi$ is not constant, we can take
$(\xi,\mu)\in\gra\partial\phi$ such that $\mu\neq 0$. By 
\eqref{e:subdiff},
\begin{equation}
\label{e:b3}
(\forall\eta\in\RR)\quad(\eta-\xi)\mu+\phi(\xi)\leq\phi(\eta).
\end{equation}
In particular, for $\eta<\xi$, since $\phi$ is increasing,
\eqref{e:b3} yields $\mu>0$. In turn, $\phi(\eta)\geq
(\eta-\xi)\mu+\phi(\xi)\uparrow\pinf$ when $\eta\to\pinf$ and we
deduce from \cite[Proposition~3.7]{Thib94} that $\phi\circ f$ is
lower semicontinuous. Finally, properness follows from \ref{p:1i}.

\ref{p:1iv}:
It follows from \cite[Proposition~4.11ii)]{Thib94},
Lemma~\ref{l:3}\ref{l:3ii}, and Lemma~\ref{l:66}\ref{l:66ii} that
\begin{align}
\label{e:thib94d}
(\phi\circ f)^*(x^*)
&=\min_{\xi^*\in\RP}\big(\phi^*(\xi^*)+(\xi^*\rocky f)^*(x^*)\big)
\nonumber\\
&=\min_{\xi^*\in\RR}\big(\phi^*(\xi^*)+\widetilde{f^*}(x^*,\xi^*)
\big).
\end{align}

\ref{p:1iii-}: This follows from Lemma~\ref{l:66}\ref{l:66iii} 
and \cite[Proposition~4.4]{Thib94}.

\ref{p:1iii}: This follows from Lemma~\ref{l:66}\ref{l:66iii} 
and \cite[Proposition~4.11i)]{Thib94}.

\ref{p:1v-b}$\Leftrightarrow$\ref{p:1v-a}: Since $f(x)\in\RR$,
Lemma~\ref{l:6}\ref{l:6i}, Lemma~\ref{l:66}\ref{l:66i}, and 
Lemma~\ref{l:2}\ref{l:2i} yield
\begin{multline}
(\phi\circ f)(x)+\phi^*(\xi^*)+(\xi^*\rocky 
f)^*(x^*)=\pair{x}{x^*}\\
\Leftrightarrow\quad 
(\xi^*\rocky f)(x)=\xi^* f(x)\leq \phi^*(\xi^*)+\phi(f(x))
=\pair{x}{x^*}-(\xi^*\rocky f)^*(x^*)
\leq (\xi^*\rocky f)(x).
\end{multline}

\ref{p:1v-b}$\Leftrightarrow$\ref{p:1v-c}: 
It follows from \ref{p:1iv} that
\begin{multline}
(\phi\circ f)(x)+\phi^*(\xi^*)+(\xi^*\rocky 
f)^*(x^*)=\pair{x}{x^*}\\
\Leftrightarrow\quad 
(\phi\circ f)^*(x^*)\leq \phi^*(\xi^*)+(\xi^*\rocky f)^*(x^*)=
\pair{x}{x^*}-(\phi\circ f)(x)\leq (\phi\circ f)^*(x^*).
\end{multline}

\ref{p:1v}: 
Note that \ref{p:1iv} and Lemma~\ref{l:66}\ref{l:66ii}
imply that  
\begin{equation}
(\phi\circ f)^*(x^*)=\min_{\eta^*\in\RR}
\big(\phi^*(\eta^*)+\widetilde{f^*}(x^*,\eta^*)\big)
=\min_{\eta^*\in\RR}
\big(\phi^*(\eta^*)+(\eta^*\,\rocky f)^*(x^*)\big).
\end{equation} 
Hence, Lemma~\ref{l:6}\ref{l:6ii} and \ref{p:1v-}
yield
\begin{align}
\begin{cases}
\xi^*\in\partial\phi\big(f(x)\big)\\
x^*\in\partial(\xi^*\rocky f)(x)
\end{cases}\;\;
&\Leftrightarrow\quad 
\begin{cases}
(\phi\circ f)(x)+\phi^*(\xi^*)=(\xi^*\rocky f)(x)\\
(\xi^*\rocky f)(x)+(\xi^*\rocky f)^*(x^*)=\pair{x}{x^*}
\end{cases}\nonumber\\
&\Leftrightarrow\quad 
\begin{cases}
(\phi\circ f)(x)+(\phi\circ f)^*(x^*)=\pair{x}{x^*}\\
\phi^*(\xi^*)+(\xi^*\rocky f)^*(x^*)=(\phi\circ f)^*(x^*)
\end{cases}\nonumber\\
&\Leftrightarrow\quad 
\begin{cases}
x^*\in\partial(\phi\circ f)(x)\\
\xi^*\in\Argmin\big(\phi^*+\widetilde{f^*}(x^*,\cdot)\big),
\end{cases}
\end{align}
which completes the proof.
\end{proof}

We now bring into play Rockafellar's perturbation theory.

\begin{definition}[\cite{Rock74}]
\label{d:74}
Let $\XX$ and $\UU$ be reflexive real Banach spaces, let 
$\varphi\colon\XX\to\RX$ be a proper function, and consider the 
primal problem
\begin{equation}
\label{e:p}
\minimize{x\in\XX}{\varphi(x)}.
\end{equation}
Let $\Phi\colon\XX\oplus\UU\to\RX$ be a perturbation of $\varphi$,
i.e., $(\forall x\in\XX)$ $\varphi(x)=\Phi(x,0)$. The Lagrangian is
\begin{equation}
\label{e:L}
\mathscr{L}_{\Phi}\colon\XX\oplus\UU^*\mapsto\RXX\colon 
(x,u^*)\mapsto\inf_{u\in\UU}\big(\Phi(x,u)-\pair{u}{u^*}\big),
\end{equation}
the dual problem is
\begin{equation}
\label{e:d}
\minimize{u^*\in\UU^*}
{\sup_{x\in\XX}\big(-\mathscr{L}_{\Phi}(x,u^*)\big)},
\end{equation}
and the Kuhn--Tucker operator is
\begin{equation}
\label{e:K}
\mathscr{K}_{\Phi}\colon\XX\oplus\UU^*\to 2^{\XX^*\oplus\,\UU}
\colon (x,u^*)\mapsto\partial
\big(\mathscr{L}_{\Phi}(\cdot,u^*)\big)(x)\times
\partial\big(-\mathscr{L}_{\Phi}(x,\cdot)\big)(u^*).
\end{equation}
\end{definition}

\begin{lemma}
\label{l:9}
Let $\XX$ and $\UU$ be reflexive real Banach spaces, let 
$\varphi\in\Gamma_0(\XX)$, let $\Phi\in\Gamma_0(\XX\oplus\UU)$ 
be a perturbation of $\varphi$, and let $\mathscr{K}_{\Phi}$
be the Kuhn--Tucker operator of \eqref{e:K}. Then the following
hold:
\begin{enumerate}
\item 
\label{l:9i-}
Let $x\in\XX$. Then
$(-\mathscr{L}_{\Phi}(x,\cdot))^*=\Phi(x,\cdot)$.
\item 
\label{l:9i}
$\mathscr{K}_{\Phi}$ is maximally monotone.
\item 
\label{l:9ii}
Suppose that $(x,u^*)\in\zer\mathscr{K}_{\Phi}$. Then $x$ solves 
\eqref{e:p} and $u^*$ solves \eqref{e:d}.
\end{enumerate}
\end{lemma}
\begin{proof}
\ref{l:9i-}: Let $u\in\UU$. Then, appealing to \eqref{e:L} and 
Lemma~\ref{l:2}\ref{l:2i}, we obtain
\begin{equation}
\big(-\mathscr{L}_{\Phi}(x,\cdot)\big)^*(u)=\sup_{u^*\in\,\UU^*}
\Big(\pair{u}{u^*}-\big(\Phi(x,\cdot)\big)^*(u^*)\Big)
=\big(\Phi(x,\cdot)\big)^{**}(u)=\Phi(x,u).
\end{equation}

\ref{l:9i}: Since $\Phi\in\Gamma_0(\XX\oplus\UU)$, it follows from
\cite[Theorem~6~and~Example~13]{Rock74} that the Lagrangian
$\mathscr{L}_{\Phi}$ is a closed proper saddle function in the
sense of \cite[Section~3]{Roc70b}. The claim therefore follows from
\cite[Theorem~3]{Roc70b}.

\ref{l:9ii}: See \cite[Theorem~15]{Rock74}.
\end{proof}

\section{Perturbation theory for nonlinear composite minimization}
\label{sec:3}

Our strategy to analyze and solve Problem~\ref{prob:1} hinges on
the perturbation function $F$ introduced in \eqref{e:F}. We
first study the nonlinear composition component of this
perturbation.

\begin{proposition}
\label{p:A}
Let $\XX$ be a reflexive real Banach space, let
$\phi\in\Gamma_0(\RR)$ be a nonconstant increasing function,
let $f\in\Gamma_0(\XX)$ be such that 
$(\dom\phi)\cap f(\dom f)\neq\emp$, and let
\begin{align}
\label{e:P}
\Psi\colon \XX\oplus\RR&\to\RX\nonumber\\
(x,\xi)&\mapsto 
\begin{cases}
\phi\big(f(x)+\xi\big),&\text{if}\;\;f(x)+\xi\in\dom\phi;\\
\pinf,&\text{if}\;\;f(x)+\xi\notin\dom\phi
\end{cases}
\end{align}
be a perturbation of $\phi\circ f$. Let $\mathscr{L}_{\Psi}$ be the
associated Lagrangian (see \eqref{e:L}) and let
$\mathscr{K}_{\Psi}$ be the associated Kuhn--Tucker operator (see
\eqref{e:K}). Then the following hold:
\begin{enumerate}
\item 
\label{p:Ai-}
$\Psi\in\Gamma_0(\XX\oplus\RR)$.
\item 
\label{p:Ai}
We have 
\begin{align}
\label{e:97}
\hspace{-8mm}\mathscr{L}_{\Psi}\colon\XX\oplus\RR
&\to\RXX\nonumber\\
(x,\xi^*)&\mapsto 
\begin{cases}
\pinf,&\text{if}\;\;x\notin \dom f;\\
(\xi^*\rocky f)(x)-\phi^*(\xi^*),&\text{if}\;\; x\in \dom 
f\,\;\text{and}\,\,\xi^*\in\dom\phi^*;\\
\minf, 
&\text{if}\;\;x\in\dom f\,\;\text{and}\,\,\xi^*\notin\dom\phi^*.
\end{cases}
\end{align}
\item 
\label{p:Aii}
We have 
\begin{align}
\label{e:Akt}
\hspace{-7mm}\mathscr{K}_{\Psi}\colon\XX\oplus\RR
&\to 2^{\XX^*\oplus\RR}\nonumber\\
(x,\xi^*)&\mapsto
\begin{cases}
\partial(\xi^*\rocky f)(x)\times\big
(\partial\phi^*(\xi^*)-f(x)\big),&\text{if}\;\;x\in\dom 
f\,\;\text{and}\,\,\xi^*\in\dom\phi^*;\\
\emp,&\text{if}\;\;x\notin\dom 
f\,\;\text{or}\,\,\xi^*\notin\dom\phi^*.
\end{cases}
\end{align}
\item 
\label{p:Aiv}
$\dom\mathscr{K}_{\Psi}\subset\dom f\times\RP$. 
\item 
\label{p:Av}
$\mathscr{K}_{\Psi}$ is maximally monotone.
\end{enumerate}
\end{proposition}
\begin{proof}
\ref{p:Ai-}: Set $\boldsymbol{f}\colon\XX\oplus\RR\to\RX\colon 
(x,\xi)\mapsto f(x)+\xi$. Then 
$\boldsymbol{f}\in\Gamma_0(\XX\oplus\RR)$ and \eqref{e:c1} implies 
that $\Psi=\phi\circ\boldsymbol{f}$. 
Thus, since $\emp\neq (\dom\phi)\cap f(\dom f)\subset 
(\dom\phi)\cap\boldsymbol{f}(\dom\boldsymbol{f})$, the result
follows from Proposition~\ref{p:1}\ref{p:1ii}.

\ref{p:Ai}: It follows from \eqref{e:L} that, for every 
$(x,\xi^*)\in\XX\times\RR$, 
\begin{align} 
\mathscr{L}_{\Psi}(x,\xi^*)
&=\inf_{\xi\in\RR}\big(\Psi(x,\xi)-\xi\xi^*\big)\nonumber\\
&=
\begin{cases}
-\sup\limits_{\xi\in(\dom\phi)-f(x)}\Big(\xi\xi^*-
\phi\big(f(x)+\xi\big)\Big),&\text{if}\;\;x\in\dom f;\\
\pinf,&\text{if}\;\;x\notin\dom f;\\
\end{cases}\nonumber\\
&=
\begin{cases}
\xi^*f(x)-\sup\limits_{\xi+f(x)\in\dom\phi}\Big(\big(\xi+f(x)\big)
\xi^*-\phi\big(f(x)+\xi\big)\Big),&\text{if}\;\;x\in\dom f;\\
\pinf,&\text{if}\;\;x\notin\dom f;\\
\end{cases}\nonumber\\
&=
\begin{cases}
\xi^*f(x)-\phi^*(\xi^*),&\text{if}\;\; x\in\dom f;\\
\pinf,&\text{if}\;\;x\notin\dom f.
\end{cases}
\end{align}
In view of Lemma~\ref{l:3}\ref{l:3ii} and \eqref{e:66}, we obtain
\eqref{e:97}.

\ref{p:Aii}:
Let $x\in\XX$ and $\xi^*\in\RR$.
If $x\notin\dom f$, then \ref{p:Ai} yields
$\mathscr{L}_{\Psi}(x,\xi^*)=\pinf$ and hence 
$\mathscr{K}_{\Psi}(x,\xi^*)=\emp$ in view of \eqref{e:K}.
Similarly, if $\xi^*\notin\dom\phi^*$, then 
$-\mathscr{L}_{\Psi}(x,\xi^*)=\pinf$ and hence
$\mathscr{K}_{\Psi}(x,\xi^*)=\emp$.
Now suppose that $x\in\dom f$ and $\xi^*\in\dom\phi^*$. 
Then it follows from \ref{p:Ai} and Lemma~\ref{l:3}\ref{l:3ii} 
that $\partial(\mathscr{L}_{\Psi}(\cdot,\xi^*))
=\partial(\xi^*\rocky f)$. Moreover, for every $\xi\in\RR$, we
derive from Lemma~\ref{l:6}\ref{l:6ii}, Lemma~\ref{l:9}\ref{l:9i-},
\ref{p:Ai}, Lemma~\ref{l:2}\ref{l:2i}, and \eqref{e:P} that
\begin{align}
\xi\in\partial\big(-\mathscr{L}_{\Psi}(x,\cdot)\big)
(\xi^*)\;\;&\Leftrightarrow\;\;-\mathscr{L}_{\Psi}(x,\xi^*)+
\big(-\mathscr{L}_{\Psi}(x,\cdot)\big)^*(\xi)=\xi\xi^*\nonumber\\
&\Leftrightarrow\;\; 
\phi^*(\xi^*)+\phi\big(f(x)+\xi\big)=\big(f(x)+\xi\big)\xi^*
\nonumber\\
&\Leftrightarrow\;\;f(x)+\xi\in\partial\phi^*(\xi^*).
\end{align}
Altogether, this verifies that \eqref{e:K} assumes the form
announced in \eqref{e:Akt}. 

\ref{p:Aiv}: It follows from \eqref{e:Akt} and
Lemma~\ref{l:3}\ref{l:3ii} that
$\dom\mathscr{K}_{\Psi}
\subset\dom f\times\dom\partial\phi^*
\subset\dom f\times\dom\phi^*
\subset\dom f\times\RP$. 

\ref{p:Av}: This is a consequence of \ref{p:Ai-} and
Lemma~\ref{l:9}\ref{l:9i}.
\end{proof}

Next, we provide a characterization of the solutions to
Problem~\ref{prob:1} in terms of a monotone inclusion problem in
$\XX\oplus\RR$.

\begin{proposition}
\label{p:14}
Consider the setting of Problem~\ref{prob:1} and the Kuhn--Tucker
operator $\mathscr{K}_{\Psi}$ of Proposition~\ref{p:A}\ref{p:Aii}.
Set
\begin{equation}
\label{e:B}
\boldsymbol{B}\colon\XX\oplus\RR\to 2^{\XX^*\oplus\RR}
\colon (x,\xi^*)\mapsto \Big(\big(L^*\circ
\big((\partial g)\infconv(\partial\ell)\big)\circ L\big)(x)
+\partial h(x)\Big)\times\{0\}.
\end{equation}
Then the following hold:
\begin{enumerate}
\item
\label{p:14ii}
$\boldsymbol{B}$ is maximally monotone.
\item
\label{p:14iii}
$\mathscr{P}=\bigcup_{\xi^*\in\RR}\menge{x\in\XX}
{(x,\xi^*)\in\zer(\mathscr{K}_{\Psi}+\boldsymbol{B})}$.
\end{enumerate}
\end{proposition}
\begin{proof}
Lemma~\ref{l:2}\ref{l:2i}, Lemma~\ref{l:1}\ref{l:1i}, and 
assumption~\ref{a:1ii} in Problem~\ref{prob:1} yield
\begin{equation}
\label{e:trust}
g\infconv\ell\in\Gamma_0(\YY). 
\end{equation}

\ref{p:14ii}:
Since assumption~\ref{a:1iv} in Problem~\ref{prob:1} implies that 
$0\in\sri(L(\dom h)-\dom (g\infconv \ell))$, we derive from
\eqref{e:trust} that $(g\infconv\ell)\circ
L+h\in\Gamma_0(\XX)$ and it follows from Lemma~\ref{l:4},
\cite[Theorem~2.8.3(vii)]{Zali02}, and Lemma~\ref{l:2}\ref{l:2iii}
that 
\begin{equation}
\label{e:h7lb}
L^*\circ\big((\partial g)\infconv(\partial\ell)\big)\circ
L+\partial h
=L^*\circ\partial (g\infconv\ell)\circ L+\partial h
=\partial\big((g\infconv\ell)\circ L+h\big) 
\end{equation}
is maximally monotone. Thus, $\boldsymbol{B}$ is maximally
monotone.

\ref{p:14iii}: 
It follows from assumption~\ref{a:1i} in Problem~\ref{prob:1} 
and Proposition~\ref{p:1}\ref{p:1ii} that 
$\phi\circ f\in\Gamma_0(\XX)$,
and therefore from assumption~\ref{a:1iii} that
$(\phi\circ f)+h\in\Gamma_0(\XX)$. Altogether, using
Lemma~\ref{l:2}\ref{l:2iv},
\eqref{e:trust}, assumptions~\ref{a:1i}--\ref{a:1iv}, 
\cite[Theorem~2.8.3(vii)]{Zali02}, 
Proposition~\ref{p:1}\ref{p:1iii}, 
Lemma~\ref{l:4}, and Lemma~\ref{l:2}\ref{l:2ii}, we obtain 
\begin{eqnarray}
x\in\mathscr{P}\quad
&\Leftrightarrow&\quad 
0\in\partial\big(\phi\circ f+(g\infconv\ell)\circ L+h\big)(x)
\nonumber\\
&\Leftrightarrow&\quad 
\big(\exi\xi^*\in\partial\phi\big(f(x)\big)\big)
\quad 0\in\partial(\xi^*\rocky f)(x)+L^*\big(\partial 
(g\infconv\ell)(Lx)\big)+\partial h(x)\nonumber\\
\quad&\Leftrightarrow&\quad (\exi\xi^*\in\RR)\quad 
\begin{cases}
0\in\partial(\xi^*\rocky f)(x)+ 
L^*\Big(\big((\partial g)\infconv(\partial\ell)\big)(Lx)\Big)
+\partial h(x)\\
0\in\partial\phi^*(\xi^*)-f(x).
\end{cases}
\end{eqnarray}
Thus, the result follows from Proposition~\ref{p:A}\ref{p:Aii}
and \eqref{e:B}.
\end{proof}

Following the abstract perturbation theory of \cite{Rock74}
outlined in Definition~\ref{d:74}, the perturbation \eqref{e:F}
allows us to construct a Lagrangian for Problem~\ref{prob:1}, a
dual problem, and a Kuhn--Tucker operator that will provide
solutions to Problem~\ref{prob:1} and its dual. 
This program is described in the next theorem.

\begin{theorem}
\label{t:1}
Consider the setting of Problem~\ref{prob:1}, set 
$D=\dom f\cap\dom h$ and set 
$V=\dom\phi^*\times(\dom g^*\cap\dom\ell^*)$. Then the following
hold with respect to the perturbation function $F$ of
\eqref{e:F}:
\begin{enumerate}
\item 
\label{t:1i}
The Lagrangian is
\begin{align}
\label{e:L1}
\hspace{-8mm}
\mathscr{L}_F\colon\XX\times\RR\times\YY^*
&\to\RXX\nonumber\\
(x,\xi^*,y^*)&\mapsto 
\begin{cases}
\pinf,\hspace{28mm}\text{if}\;\; x\notin D;\\
(\xi^*\rocky f)(x)+
h(x)+\pair{Lx}{y^*}-\phi^*(\xi^*)-g^*(y^*)-\ell^*(y^*),\\
\hspace{37mm}\text{if}\;\; x\in D\;\;\text{and}\;\;
(\xi^*,y^*)\in V;\\
\minf,\hspace{28.4mm}\text{if}\;\; 
x\in D\;\;\text{and}\;\;(\xi^*,y^*)\notin V.
\end{cases}
\end{align}
\item 
\label{t:1iii}
The dual problem is
\begin{equation}
\label{e:p2}
\minimize{(\xi^*,y^*)\in\RR\times\YY^*}
{\phi^*(\xi^*)+\big(h^*\infconv 
\widetilde{f^*}(\cdot,\xi^*)\big)(-L^*y^*)+g^*(y^*)+\ell^*(y^*)}.
\end{equation}
\item 
\label{t:1ii}
The Kuhn--Tucker operator is 
\begin{align}
\label{e:kt}
\hspace{-7mm}\mathscr{K}_F\colon\XX\times\RR\times\YY^*
&\to2^{\XX^*\times\RR\times\YY}\nonumber\\
(x,\xi^*,y^*)&\mapsto
\begin{cases}
\big(\partial(\xi^*\rocky f)(x)+\partial h(x)+L^*y^*\big)\times 
\big(\partial\phi^*(\xi^*)-f(x)\big)\\
\hspace{5mm}\times \big(\partial g^*(y^*)
+\partial \ell^*(y^*)-Lx\big),&\hspace{-25mm}\text{if}\;\; x\in 
D\;\;\text{and}\;\;(\xi^*,y^*)\in V;\\
\emp,&\hspace{-25mm}\text{if}\;\; x\notin 
D\;\;\text{or}\;\;(\xi^*,y^*)\notin V.
\end{cases}
\end{align}
\item 
\label{t:1v}
Let $\mathscr{D}$ be the set of solutions to \eqref{e:p2}. Then
$\zer\mathscr{K}_F\subset\mathscr{P}\times\mathscr{D}$. 
\item 
\label{t:1vi}
$\mathscr{P}=\displaystyle{\bigcup_{(\xi^*,y^*)\in\RR\times\YY^*}}
\menge{x\in\XX}{(x,\xi^*,y^*)\in\zer\mathscr{K}_F}$.
\end{enumerate}
\end{theorem}
\begin{proof}
Note that assumptions~\ref{a:1i}--\ref{a:1iv} imply that 
$D=\dom f\cap\dom h\neq\emp$ and 
$V=\dom\phi^*\times(\dom g^*\cap\dom\ell^*)\neq\emp$.

\ref{t:1i}:
Let $(x,\xi^*,y^*)\in\XX\times\RR\times\YY^*$. In view of
\eqref{e:L}, 
\begin{equation}
\label{e:44}
\mathscr{L}_F(x,\xi^*,y^*)
=\inf_{(\xi,y)\in\RR\oplus\YY}\big(F(x,\xi,y)
-\xi\xi^*-\pair{y}{y^*}\big).
\end{equation}
If $x\notin D$, \eqref{e:F} implies 
that $F(x,\xi,y)=\pinf$ and, therefore, that
$\mathscr{L}_F(x,\xi^*,y^*)=\pinf$. Now suppose that 
$x\in D$. 
Then \eqref{e:44} and \eqref{e:F} yield
\begin{align}
\mathscr{L}_F(x,\xi^*,y^*)
&=h(x)+\inf_{\xi\in(\dom\phi)-f(x)}
\big(\phi(f(x)+\xi)-\xi\xi^*\big)+\inf_{y\in\YY}
\big((g\infconv\ell)(Lx+y)-\pair{y}{y^*}\big)\nonumber\\
&=\xi^*f(x)+h(x)+\pair{Lx}{y^*}-\sup_{f(x)+\xi\in\dom\phi}
\big((f(x)+\xi)\xi^*-\phi(f(x)+\xi)\big)\nonumber\\
&\hspace{53mm}-\sup_{y\in\YY}
\big(\pair{Lx+y}{y^*}-(g\infconv\ell)(Lx+y)\big)\nonumber\\
&=(\xi^*\rocky f)(x)+h(x)+\pair{Lx}{y^*}-\phi^*(\xi^*)-
(g\infconv\ell)^*(y^*)\nonumber\\
&=(\xi^*\rocky f)(x)+h(x)+\pair{Lx}{y^*}-\phi^*(\xi^*)-
g^*(y^*)-\ell^*(y^*).
\end{align}

\ref{t:1iii}: 
First, for every $(\xi^*,y^*)\in(\RR\times\YY^*)\smallsetminus V$,
we derive from \ref{t:1i} that 
$\sup_{x\in\XX}\big(-\mathscr{L}_F(x,\xi^*,y^*)\big)=\pinf$. Next,
for every $(\xi^*,y^*)\in V$, we derive from 
Lemma~\ref{l:1}\ref{l:1i} and
Lemma~\ref{l:66}\ref{l:66ii} that 
\begin{align}
\label{e:79}
\sup_{x\in\XX}\big(-\mathscr{L}_F(x,\xi^*,y^*)\big)&=
\phi^*(\xi^*)+
g^*(y^*)+\ell^*(y^*)+\sup_{x\in D}\big(\pair{x}{-L^*y^*}
-(h+\xi^* f)(x)\big)\nonumber\\
&=\phi^*(\xi^*)+
g^*(y^*)+\ell^*(y^*)+(h+\xi^*\rocky f)^*(-L^*y^*)\nonumber\\
&=\phi^*(\xi^*)+g^*(y^*)+\ell^*(y^*)+\big(h^*\infconv 
\widetilde{f^*}(\cdot,\xi^*)\big)(-L^*y^*).
\end{align}
According to \eqref{e:d}, the dual problem is
\begin{align}
\minimize{(\xi^*,y^*)\in\RR\times\YY^*}
{\sup_{x\in\XX}\big(-\mathscr{L}_F(x,\xi^*,y^*)\big)}
\end{align}
which, in view of \eqref{e:79}, is precisely \eqref{e:p2}.

\ref{t:1ii}: 
Let $(x,\xi^*,y^*)\in\XX\times\RR\times\YY^*$. If $x\notin D$, 
then it results from \ref{t:1i} that 
$\mathscr{L}_F(x,\xi^*,y^*)=\pinf$ and therefore 
$\mathscr{K}_F(x,\xi^*,y^*)=\emp$. If $(\xi^*,y^*)\notin V$, 
then $-\mathscr{L}_F (x,\xi^*,y^*)=\pinf$ and therefore 
$\mathscr{K}_F (x,\xi^*,y^*)=\emp$. Now suppose that
$x\in D$ and $(\xi^*,y^*)\in V$, and note that 
\begin{equation}
\mathscr{L}_F (x,\xi^*,y^*)=\mathscr{L}_{\Psi} (x,\xi^*)+h(x)+
\pair{x}{L^*y^*}-g^*(y^*)-\ell^*(y^*),
\end{equation}
where $\mathscr{L}_{\Psi}$ is defined in 
Proposition~\ref{p:A}\ref{p:Ai}. Therefore, 
assumption \ref{a:1iii} and Lemma~\ref{l:1}\ref{l:1ii} yield
\begin{equation}
\partial\big(\mathscr{L}_F (\cdot,\xi^*,y^*)\big)
=\partial\big(\mathscr{L}_{\Psi} (\cdot,\xi^*)\big)+\partial h+
L^*y^*,
\end{equation}
while assumption \ref{a:1ii} and Lemma~\ref{l:1}\ref{l:1ii} imply
that
\begin{equation}
\partial\big(-\mathscr{L}_F (x,\cdot,\cdot)\big)\colon (\xi^*,y^*)
\mapsto\partial\big(-\mathscr{L}_{\Psi} (x,\cdot)\big)(\xi^*)
\times\big(\partial g^*(y^*)+\partial\ell^*(y^*)-Lx\big).
\end{equation}
Thus, \eqref{e:kt} follows from \eqref{e:K} and 
Proposition~\ref{p:A}\ref{p:Aii}.

\ref{t:1v}: 
This follows from \ref{t:1ii} and Lemma~\ref{l:9}\ref{l:9ii}.

\ref{t:1vi}: Suppose that $x\in\mathscr{P}$. Then it follows from
Proposition~\ref{p:14}\ref{p:14iii} that there exists 
$\xi^*\in\RR$ such that 
\begin{equation}
\begin{cases}
0\in\partial(\xi^*\rocky f)(x)+ 
L^*\Big(\big((\partial g)\infconv(\partial\ell)\big)(Lx)\Big)
+\partial h(x)\\
0\in\partial\phi^*(\xi^*)-f(x).
\end{cases}
\end{equation}
Thus, Lemma~\ref{l:2}\ref{l:2ii} and \eqref{e:94} guarantee
the existence of $(\xi^*,y^*)\in\RR\times\YY^*$ such that 
\begin{equation}
\label{e:y2}
\begin{cases}
0\in\partial(\xi^*\rocky f)(x)+\partial h(x)+L^*y^*\\
0\in\partial\phi^*(\xi^*)-f(x)\\
0\in\partial g^*(y^*)+\partial\ell^*(y^*)-Lx,
\end{cases}
\end{equation}
which shows that $(x,\xi^*,y^*)\in\zer\mathscr{K}_F$.
Conversely, if $(x,\xi^*,y^*)\in\zer\mathscr{K}_F$, then
\ref{t:1ii} yields $x\in\mathscr{P}$.
\end{proof}

\begin{remark} 
In Theorem~\ref{t:1}, suppose that $\phi\colon\xi\mapsto\xi$,
$h=0$, and $\ell=\iota_{\{0\}}$. Then $\phi^*=\iota_{\{1\}}$, the
primal problem \eqref{e:p1} reduces to 
\begin{equation}
\label{e:flipper}
\minimize{x\in\XX}{f(x)+g(Lx)},
\end{equation}
the perturbation function \eqref{e:F} to
$(x,y)\mapsto f(x)+g(Lx+y)$,
the Lagrangian \eqref{e:L1} to 
\begin{equation}
(x,y^*)\mapsto 
\begin{cases}
\pinf,&\text{if}\;\;x\notin\dom f;\\
f(x)+\pair{Lx}{y^*}-g^*(y^*),&\text{if}\;\;
x\in\dom f\;\text{and}\;y^*\in\dom g^*;\\
\minf,&\text{if}\;\;
x\in\dom f\;\text{and}\;y^*\notin\dom g^*,
\end{cases}
\end{equation}
the dual problem \eqref{e:p2} to 
\begin{equation}
\minimize{y^*\in\YY^*}{f^*(-L^*y^*)+g^*(y^*)},
\end{equation}
and the Kuhn--Tucker operator \eqref{e:kt} to 
$(x,y^*)\mapsto\big(\partial f(x)+L^*y^*\big)\times
\big(\partial g^*(y^*)-Lx\big)$.
We thus recover the classical Lagrangian, Fenchel--Rockafellar
dual, and Kuhn--Tucker operator associated to \eqref{e:flipper}
\cite[Examples~11\,\&\,11']{Rock74}.
\end{remark}

\begin{remark}
In Theorem~\ref{t:1}, suppose that $g=\ell=0$. Then 
\eqref{e:p1} reduces to 
\begin{equation}
\label{e:botp}
\minimize{x\in\XX}{(\phi\circ f)(x)+h(x)}
\end{equation}
and the perturbation function \eqref{e:F} becomes
\begin{equation}
\label{e:pertbot}
(x,\xi)\mapsto 
\begin{cases}
\phi(f(x)+\xi)+h(x),&\text{if}\;\; f(x)+\xi\in\dom\phi;\\
\pinf, &\text{if}\;\; f(x)+\xi\notin\dom\phi.
\end{cases}
\end{equation}
In turn, the Lagrangian \eqref{e:L1} is
\begin{equation}
(x,\xi^*)\mapsto 
\begin{cases}
\pinf,&\text{if}\;\;x\notin\dom f\;\;\text{or}\;\;x\notin\dom h;\\
h(x)+(\xi^*\rocky f)(x)-\phi^*(\xi^*),&\text{if}\;\;
x\in\dom f\cap\dom h\;\text{and}\;\xi^*\in\dom \phi^*;\\
\minf,&\text{if}\;\;
x\in\dom f\cap\dom h\;\text{and}\;\xi^*\notin\dom \phi^*,
\end{cases}
\end{equation}
the dual problem \eqref{e:p2} is
\begin{equation}
\label{e:dualbot}
\minimize{\xi^*\in\RR}{\phi^*(\xi^*)+
\big(h^*\infconv\widetilde{f^*}
(\cdot,\xi^*)\big)(0)},
\end{equation}
and the Kuhn--Tucker operator \eqref{e:kt} is
\begin{equation}
(x,\xi^*)\mapsto
\begin{cases}
\big(\partial (\xi^*\rocky f)(x)+\partial 
h(x)\big)\times
\big(\partial \phi^*(\xi^*)-f(x)\big),
&\text{if}\;\; x\in\dom f\cap\dom 
h\;\;\text{and}\;\; \xi^*\in\dom\phi^*;\\
\emp,&\text{if}\;\;x\notin\dom f\cap\dom 
h\;\;\text{or}\;\; 
\xi^*\notin\dom\phi^*.
\end{cases}
\end{equation}
The perturbation function \eqref{e:pertbot} was introduced in
\cite{Thib94}. It is employed in \cite[Section~I.4]{Botr10} to
obtain \eqref{e:dualbot} and in \cite{Rock22} in the context of
augmented Lagrangian formulations.
\end{remark}

Our next topic is the splitting of the Kuhn--Tucker operator
of Theorem~\ref{t:1}\ref{t:1ii} into elementary components. This
decomposition will pave the way to the numerical solution of
Problem~\ref{prob:1}.

\begin{proposition}
\label{p:6}
Consider the setting of Problem~\ref{prob:1}, set 
$\XXX=\XX\oplus\RR\oplus\YY^*$, 
let $\mathscr{K}_{\Psi}$ be the Kuhn--Tucker
operator of Proposition~\ref{p:A}\ref{p:Aii}, 
and let $\mathscr{K}_F$ be 
the Kuhn--Tucker operator of Theorem~\ref{t:1}\ref{t:1ii}. Define
\begin{equation}
\label{e:xena}
\begin{cases}
\boldsymbol{M}\colon\XXX\to
2^{\XXX^*}\colon(x,\xi^*,y^*)\mapsto
\mathscr{K}_{\Psi}(x,\xi^*)\times\partial g^*(y^*)\\
\boldsymbol{S}\colon\XXX\to
\XXX^*\colon(x,\xi^*,y^*)\mapsto (L^*y^*,0,-Lx)\\
\boldsymbol{G}\colon\XXX\to
2^{\XXX^*}\colon(x,\xi^*,y^*)\mapsto\partial h(x)
\times\{0\}\times\partial\ell^*(y^*).
\end{cases}
\end{equation}
Then the following hold:
\begin{enumerate}
\item 
\label{p:6i-}
$\mathscr{K}_F=\boldsymbol{M}+\boldsymbol{S}+\boldsymbol{G}$.
\item 
\label{p:6ii}
$\boldsymbol{M}$ is maximally monotone.
\item 
\label{p:6ii+}
$\boldsymbol{S}$ is bounded, linear, and skew in the sense that
$(\forall\boldsymbol{x}\in\XXX)$ 
$\pair{\boldsymbol{x}}{\boldsymbol{S}\boldsymbol{x}}=0$. 
In addition, $\|\boldsymbol{S}\|=\|L\|$.
\item 
\label{p:6ii++}
$\boldsymbol{G}$ is maximally monotone.
\item 
\label{p:6ii+++}
Suppose that $h$ and $\ell^*$ are differentiable and that their
gradients are Lipschitzian with constants $\delta\in\RPP$ and
$\vartheta\in\RPP$, respectively. Then $\boldsymbol{G}$ is 
$\max\{\delta,\vartheta\}$-Lipschitzian and 
$\min\{1/\delta,1/\vartheta\}$-cocoercive.
\end{enumerate}
\end{proposition}
\begin{proof}
Define 
\begin{equation}
\label{e:defh}
\boldsymbol{h}\colon\XXX\to\RX\colon(x,\xi^*,y^*)
\mapsto h(x)+\ell^*(y^*).
\end{equation}

\ref{p:6i-}: Combine \eqref{e:Akt}, \eqref{e:kt}, and
\eqref{e:xena}.

\ref{p:6ii}: 
As seen in Proposition~\ref{p:A}\ref{p:Av}, $\mathscr{K}_{\Psi}$
is maximally monotone. Therefore, it follows from
Lemma~\ref{l:2}\ref{l:2i} and Lemma~\ref{l:2}\ref{l:2iii} that
$\boldsymbol{M}$ is maximally monotone. 

\ref{p:6ii+}: 
The first three claims are clear. Now let
$\boldsymbol{x}=(x,\xi^*,y^*)\in\XXX$. Then \eqref{e:n} yields
\begin{equation}
\|\boldsymbol{S}\boldsymbol{x}\|^2=\|L^*y^*\|^2+\|Lx\|^2\leq
\|L\|^2\|\boldsymbol{x}\|^2,
\end{equation}
which shows that $\|\boldsymbol{S}\|\leq\|L\|$.
On the other hand, suppose that $\|x\|\leq 1$ and that 
$(\xi^*,y^*)=(0,0)$. Then
$\|Lx\|=\|\boldsymbol{S}\boldsymbol{x}\|\leq\|\boldsymbol{S}\|$
and, therefore, $\|L\|\leq \|\boldsymbol{S}\|$.

\ref{p:6ii++}: Since $\boldsymbol{h}\in\Gamma_0(\XXX)$, 
Lemma~\ref{l:2}\ref{l:2iii} asserts that 
$\boldsymbol{G}=\partial\boldsymbol{h}$ is maximally monotone.

\ref{p:6ii+++}: By \eqref{e:defh}, $\boldsymbol{h}$ is 
differentiable and $\boldsymbol{G}=\nabla\boldsymbol{h}\colon 
(x,\xi^*,y^*)\mapsto(\nabla h(x),0,\nabla\ell^*(y^*))$ 
is $\max\{\delta,\vartheta\}$-Lipschitzian, which makes it
$\min\{1/\delta,1/\vartheta\}$-cocoercive by
\cite[Corollaire~10]{Bail77}.
\end{proof}

\section{Proximal analysis and solution methods}
\label{sec:4}

We turn our attention to the design of algorithms for solving the
primal problem \eqref{e:p1} and its dual \eqref{e:p2} in the case
when $\XX$ and $\YY$ are real Hilbert spaces.
Theorem~\ref{t:1}\ref{t:1v} opens a path towards this objective by
seeking a zero of the Kuhn--Tucker operator $\mathscr{K}_F$ of
\eqref{e:kt}, i.e., in view of Proposition~\ref{p:6}\ref{p:6i-}, of
the sum of the maximally monotone operators $\boldsymbol{M}$,
$\boldsymbol{S}$, and $\boldsymbol{G}$. This can be achieved by
splitting methods which activate these operators separately and
involve resolvents. 

\subsection{Resolvent computations}
\label{sec:41}
We start with the computation of the resolvent of the Kuhn--Tucker
operator $\mathscr{K}_{\Psi}$.

\begin{proposition}
\label{p:71}
Let $\XX$ be a real Hilbert space, let $f\in\Gamma_0(\XX)$, and let
$\phi\in\Gamma_0(\RR)$ be a nonconstant increasing function such
that $(\inte\dom\phi)\cap f(\dom f)\neq\emp$. 
Let $\mathscr{K}_{\Psi}$ be the Kuhn--Tucker operator defined in
Proposition~\ref{p:A}\ref{p:Aii}, let $x\in\XX$, let $\xi^*\in\RR$,
and let $\gamma\in\RPP$. Then there exists a unique real number
$\omega\in\RP$ such that 
\begin{equation}
\label{e:o}
\omega=\prox_{\gamma\phi^*}\big(\xi^*+
\gamma f(\prox_{\gamma(\omega\rocky f)}x)\big).
\end{equation}
Moreover,
\begin{equation}
\label{e:7ii}
J_{\gamma\mathscr{K}_{\Psi}}(x,\xi^*)=
\big(\prox_{\gamma(\omega\rocky f)}x,\omega\big).
\end{equation}
\end{proposition}
\begin{proof}
We deduce from Proposition~\ref{p:A}\ref{p:Av} and
\cite[Proposition~23.8]{Livre1} that
$J_{\mathscr{K}_{\Psi}}$ is single-valued on $\XX\times\RR$.
Now let $(p,\omega)\in\XX\times\RR$. Then, by \eqref{e:Akt},
Proposition~\ref{p:A}\ref{p:Aiv}, and Lemma~\ref{l:66}\ref{l:66i},
\begin{eqnarray}
\label{e:incres}
(p,\omega)=J_{\gamma\mathscr{K}_{\Psi}}(x,\xi^*)
&\Leftrightarrow&
(x,\xi^*)\in 
(p,\omega)+\gamma\mathscr{K}_{\Psi}(p,\omega)\nonumber\\
&\Leftrightarrow&
\begin{cases}
p\in\dom f\;\text{and}\;\omega\in\RP\\
x\in p+\gamma\partial(\omega\rocky f)(p)\\
\xi^*\in\omega+\gamma\partial\phi^*(\omega)-\gamma f(p)
\end{cases}\nonumber\\
&\Leftrightarrow&
\begin{cases}
p\in\dom f\;\text{and}\;\omega\in\RP\\
p=\prox_{\gamma(\omega\rocky f)}x\\
\omega=\prox_{\gamma\phi^*}\big(\xi^*+\gamma f(p)\big),
\end{cases}
\end{eqnarray}
which proves the assertion.
\end{proof}

\begin{remark}
\label{r:K} 
In the setting of Proposition~\ref{p:71}, let us provide details on
the computation of the unique solution $\omega\in\RP$ to
\eqref{e:o} and of $J_{\gamma\mathscr{K}_{\Psi}}(x,\xi^*)$. These
computations can be performed by testing the membership of
$\xi^*/\gamma+f(\proj_{\cdom f}x)$ in $\Argmin\phi$. Using 
\cite[Proposition~12.29 and Theorem~14.3(ii)]{Livre1} as well as 
Lemma~\ref{l:66}\ref{l:66iv}, we first observe that
\begin{align}
\label{e:e45}
\xi^*/\gamma+f(\proj_{\cdom f}x)\in\Argmin\phi\quad 
&\Leftrightarrow\quad 
\xi^*/\gamma+f(\proj_{\cdom f}x)\in\Argmin\big(\phi/\gamma\big)
\nonumber\\
&\Leftrightarrow\quad 
\big(\Id-\prox_{\phi/\gamma}\big)\big(\xi^*/\gamma+
f(\proj_{\cdom f}x)\big)=0\nonumber\\
&\Leftrightarrow\quad 
\gamma^{-1}\,\prox_{\gamma\phi^*}\big(\xi^*+
\gamma f(\proj_{\cdom f}x)\big)=0\nonumber\\
&\Leftrightarrow\quad 
\prox_{\gamma\phi^*}\big(\xi^*+
\gamma f(\prox_{\gamma(0\rocky f)}x)\big)=0\nonumber\\
&\Leftrightarrow\quad 
\omega=0.
\end{align}
\begin{itemize}
\item 
\label{r:Ki} 
{\bfseries Case 1:} 
$\xi^*/\gamma+f(\proj_{\cdom f}x)\in\Argmin\phi$.
Then \eqref{e:e45} implies that $\omega=0$ and therefore 
\eqref{e:7ii} 
yields
\begin{equation}
\label{e:J1}
J_{\gamma\mathscr{K}_{\Psi}}(x,\xi^*)=
\big(\proj_{\cdom f}x,0\big).
\end{equation}
\item 
\label{r:Ki2} 
{\bfseries Case 2:} $
\xi^*/\gamma+f(\proj_{\cdom f}x)\notin\Argmin\phi$.
Then it follows from \eqref{e:e45} that $\omega>0$.
Now define
\begin{equation}
\label{e:defT}
T\colon\RPP\to\RR\colon\tau\mapsto\tau-
\prox_{\gamma\phi^*}\big(\xi^*+\gamma f
(\prox_{\gamma\tau f}x)\big).
\end{equation}
On the one hand, we deduce from \cite[Lemma~3.27]{Atto84} and
\cite[Propositions~12.27 and 24.31]{Livre1} that
$T$ is continuous and strictly increasing. On the other hand,
\eqref{e:o} and \eqref{e:66} imply that 
$\omega\in\zer T$, hence $\zer T=\{\omega\}$. Thus, 
$\omega$ can be computed by standard one-dimensional root-finding
numerical schemes \cite[Chapter~9]{Pres07}. In turn, \eqref{e:7ii}
yields
\begin{equation}
\label{e:J2}
J_{\gamma\mathscr{K}_{\Psi}}(x,\xi^*)=
\big(\prox_{\gamma\omega f}x,\omega\big).
\end{equation}
\end{itemize}
\end{remark}

In the following examples, we compute the resolvent
$J_{\gamma\mathscr{K}_{\Psi}}$ of Proposition~\ref{p:71} with the
help of Remark~\ref{r:K}.

\begin{example}
\label{ex:1b}
In Example~\ref{ex:1}, suppose that $\XX$ is a real
Hilbert space, let $\gamma\in\RPP$, let $(x,\xi^*)\in\XX\times\RR$,
and let $\omega$ be the unique real number in $\RP$ such that
\begin{equation}
\label{e:scalarineq}
\omega=
\begin{cases}
0,&\text{if}\;\;\xi^*+\gamma f(\proj_{\cdom f}x)\leq 0;\\
\xi^*+\gamma f\big(\prox_{\gamma\omega f}x\big),&\text{if}\;\;
\xi^*+\gamma f(\proj_{\cdom f}x)>0.
\end{cases}
\end{equation}
Since $\phi^*=\iota_{\RP}$ and $\Argmin\phi=\RM$, 
\eqref{e:J1} and \eqref{e:J2} yield
\begin{equation}
\label{e:cconst}
J_{\gamma \mathscr{K}_{\Psi}}(x,\xi^*)
=\begin{cases}
(\proj_{\cdom f}x,0),
&\text{if}\;\;\xi^*+\gamma f(\proj_{\cdom f}x)\leq 0;\\
\big(\prox_{\gamma\omega f}x,\omega\big),
&\text{if}\;\;\xi^*+\gamma f(\proj_{\cdom f}x)>0.
\end{cases}
\end{equation}
\end{example}

\begin{example}
\label{ex:3b}
In Example~\ref{ex:3}, set $\alpha\colon\rho\mapsto\rho$, let
$\gamma\in\RPP$, let $(x,\xi^*)\in\XX\times\RR$, and let $\omega$
be the unique real number in $\RP$ such that
\begin{equation}
\omega=
\begin{cases}
0,&\text{if}\;\;(\xi^*/\gamma+f(\proj_{\cdom f}x))/\rho\in
\Argmin\theta;\\
\prox_{\gamma\rho\theta^*}\big(\xi^*+
\gamma f(\prox_{\gamma\omega f}x)\big),
&\text{if}\;\;(\xi^*/\gamma+f(\proj_{\cdom f}x))/\rho\notin
\Argmin\theta.
\end{cases}
\end{equation}
Then, by \cite[Proposition~13.23(ii)]{Livre1},
$\phi^*=\rho\theta^*$, $\Argmin\phi=\rho\,\Argmin\theta$, and it 
follows from \eqref{e:J1} and \eqref{e:J2} that 
\begin{equation}
\label{e:cconst2}
J_{\gamma \mathscr{K}_{\Psi}}(x,\xi^*)
=\begin{cases}
(\proj_{\cdom f}x,0),
&\text{if}\;\;(\xi^*/\gamma+f(\proj_{\cdom f}x))/\rho\in
\Argmin\theta;\\
\big(\prox_{\gamma\omega f}x,\omega\big),
&\text{if}\;\;(\xi^*/\gamma+f(\proj_{\cdom f}x))/\rho\notin
\Argmin\theta.
\end{cases}
\end{equation}
\end{example}

\begin{example}
\label{ex:2b}
Consider the setting of Example~\ref{ex:2}, where we suppose that 
$\XX$ is a real Hilbert space and that $\theta$ is not of the form
$\theta=\iota_{\{0\}}+\nu$ for some $\nu\in\RR$. 
Set $\rho=\sup\partial\theta(0)$,
let $\gamma\in\RPP$, and let $(x,\xi^*)\in\XX\times\RR$. Then
$\rho\in\RPX$, $\phi=\theta\circ d_{\RM}$, $\Argmin\phi=\RM$, 
and, since $\theta$ is even, \cite[Example~13.26]{Livre1} yields
$\gamma\phi^*=\gamma(\sigma_{\RM}+\theta^*\circ|\cdot|)
=\sigma_{\RM}+\gamma\theta^*\circ|\cdot|$. In turn, since 
$\theta^*$ is not constant, it follows from 
\cite[Proposition~2.2]{Nmtm09} that
\begin{equation}
\label{e:jjm1}
\prox_{\gamma\phi^*}\xi^*=
\begin{cases}
0,&\text{if}\;\;\xi^*\leq 0;\\
\xi^*,&\text{if}\;\;0<\xi^*\leq\rho;\\
\prox_{\gamma\theta^*}\xi^*,&\text{if}\;\;\xi^*>\rho.
\end{cases}
\end{equation}
Now let $\omega$ be the unique real number in $\RP$ such that
\begin{equation}
\label{e:jjm2}
\omega=
\begin{cases}
0,&\text{if}\;\;\xi^*+\gamma f(\proj_{\cdom f}x)\leq 0;\\
\prox_{\gamma\phi^*}\big(\xi^*+\gamma f(\prox_{\gamma\omega f}x)
\big),&\text{if}\;\;\xi^*+\gamma f(\proj_{\cdom f}x)>0.
\end{cases}
\end{equation}
Then \eqref{e:J1} and \eqref{e:J2} yield 
\begin{equation}
J_{\gamma \mathscr{K}_{\Psi}} (x,\xi^*)
=\begin{cases}
\big(\proj_{\cdom f}x,0\big),
&\text{if}\;\;\xi^*+\gamma f(\proj_{\cdom f}x)\leq 0;\\
\big(\prox_{\gamma\omega f}x,\omega\big),
&\text{if}\;\;\xi^*+\gamma f(\proj_{\cdom f}x)>0.
\end{cases}
\end{equation}
\end{example}

We now turn to the computation of the resolvents of
$\boldsymbol{M}$, $\boldsymbol{S}$, and $\boldsymbol{G}$
in \eqref{e:xena}.

\begin{proposition}
\label{p:7}
In Proposition~\ref{p:6}, suppose that $\XX$ and $\YY$ are real
Hilbert spaces identified with their duals. Let $x\in\XX$,
$\xi^*\in\RR$, $y^*\in\YY$, and $\gamma\in\RPP$. Then the following
hold: 
\begin{enumerate}
\item 
\label{p:7i}
$J_{\gamma\boldsymbol{M}}(x,\xi^*,y^*)=
(J_{\gamma\mathscr{K}_{\Psi}}(x,\xi^*),\prox_{\gamma g^*}y^*)$.
\item
\label{p:7iv}
$J_{\gamma\boldsymbol{S}}(x,\xi^*,y^*)=
((\Id+\,\gamma^2L^*L)^{-1}(x-\gamma L^*y^*),\xi^*,
(\Id+\,\gamma^2LL^*)^{-1}(y^*+\gamma Lx))$.
\item 
\label{p:7iii}
$J_{\gamma\boldsymbol{G}}(x,\xi^*,y^*)=
\big(\prox_{\gamma h}x,\xi^*,\prox_{\gamma \ell^*}y^*\big)$.
\end{enumerate}
\end{proposition}
\begin{proof}
\ref{p:7i}: This follows from Proposition~\ref{p:6}\ref{p:6ii}, 
Lemma~\ref{l:9}\ref{l:9i}, and \cite[Proposition~23.18]{Livre1}.

\ref{p:7iv}: This follows from \cite[Example~23.5]{Livre1}.

\ref{p:7iii}: This follows from \cite[Proposition~23.18]{Livre1}.
\end{proof}

\subsection{Algorithms}
\label{sec:42}
As an application of Theorem~\ref{t:1}\ref{t:1v},
Proposition~\ref{p:6}, and Proposition~\ref{p:7}, we now design
algorithms for solving the primal problem \eqref{e:p1} and its dual
\eqref{e:p2}. Several options can be considered to find a zero of
$\mathscr{K}_F=\boldsymbol{M}+\boldsymbol{S}+\boldsymbol{G}$
depending on the assumptions on the constituents of the problem. In
the following approach, we adopt a strategy based on the
forward-backward-half forward splitting method of \cite{Siop18}.

\begin{notation}
\label{n:K}
In the setting of Proposition~\ref{p:71}, we denote by
$\mm(\gamma,x,\xi^*)$ the unique real number $\omega\in\RP$ such 
that $\omega=\prox_{\gamma\phi^*}(\xi^*+\gamma 
f(\prox_{\gamma(\omega\rocky f)}x))$.
\end{notation}

\begin{proposition}
\label{p:2}
Consider the setting of Problem~\ref{prob:1}, let $\mathscr{D}$
be the set of solutions to the dual problem \eqref{e:p2}, and 
define $\mm$ as in Notation~\ref{n:K}. Suppose that $\XX$ and $\YY$
are real Hilbert spaces, that $\mathscr{P}\neq\emp$, that $h$ and
$\ell^*$ are differentiable, and that $\nabla h$ and $\nabla\ell^*$
are $\delta$- and $\vartheta$-Lipschitzian, respectively, for some
$(\delta,\vartheta)\in\RPP^2$. Let 
$(x_0,\xi_0^*,y^*_0)\in\XX\times\RR\times\YY$, 
let $\chi=4\max\{\delta,\vartheta\}/(1+\sqrt{1+
16\max\{\delta,\vartheta\}^2\|L\|^2})$,
let $\varepsilon\in\left]0,\chi/2\right[$,
let $(\gamma_n)_{n\in\NN}$ be a sequence in
$[\varepsilon,\chi-\varepsilon]$,
and iterate
\begin{equation}
\label{e:algo1}
\begin{array}{l}
\text{for}\;n=0,1,\ldots\\
\left\lfloor
\begin{array}{l}
z_n=x_n-\gamma_n\big(L^*y^*_n+\nabla h(x_n)\big)\\
z_n^*=y^*_n+\gamma_n\big(Lx_n-\nabla \ell^*(y^*_n)\big)\\
\xi^*_{n+1}=\mm(\gamma_n,z_n,\xi_n^*)\\
p_n=\prox_{\gamma_n(\xi^*_{n+1}\rocky f)}{z}_{n}\\
p_n^*=\prox_{\gamma_ng^*}z_n^*\\
x_{n+1}=p_n+\gamma_n L^*(y_n^*-p_n^*)\\
y^*_{n+1}=p_n^*-\gamma_n L(x_n-p_n).
\end{array}
\right.\\[2mm]
\end{array}
\end{equation}
Then $x_n\weakly\overline{x}$, $y^*_n\weakly\overline{y^*}$, and 
$\xi_n^*\to\overline{\xi^*}$, where $\overline{x}\in\mathscr{P}$
and $(\overline{\xi^*},\overline{y^*})\in\mathscr{D}$.
\end{proposition} 
\begin{proof}
Set $\XXX=\XX\oplus\RR\oplus\YY$, let
$(\boldsymbol{M},\boldsymbol{S},\boldsymbol{G})$ be as in
\eqref{e:xena}, and let $\mathscr{K}_F$ be as in \eqref{e:kt}. We
first note that Theorem~\ref{t:1}\ref{t:1vi} asserts that 
$\zer\mathscr{K}_F\neq\emp$. Further, it follows from 
Proposition~\ref{p:6}\ref{p:6ii}--\ref{p:6ii+++} that 
$\boldsymbol{M}$ is maximally monotone,
that $\boldsymbol{S}$ is monotone and
$\|L\|$-Lipschitzian, and that $\boldsymbol{G}$ is 
$(\max\{\delta,\vartheta\})^{-1}$-cocoercive. Now set, for every
$n\in\NN$,
$\boldsymbol{x}_n=(x_n,\xi_n^*,y_n^*)$,
$\boldsymbol{z}_n=(z_n,\xi_n^*,z_n^*)$,
$\boldsymbol{p}_n=(p_n,\xi_{n+1}^*,p_n^*)$,
and $\boldsymbol{q}_n=(q_n,\xi_{n+1}^*,q_n^*)$. Then, 
in view of \eqref{e:xena}, Proposition~\ref{p:71}, and  
Proposition~\ref{p:7}\ref{p:7i}, we can express \eqref{e:algo1} as
\begin{equation}
\label{e:tseng}
\begin{array}{l}
\text{for}\;n=0,1,\ldots\\
\left\lfloor
\begin{array}{l}
\boldsymbol{z}_n=\boldsymbol{x}_n-
\gamma_n(\boldsymbol{S}\boldsymbol{x}_n+
\boldsymbol{G}\boldsymbol{x}_n)\\
\boldsymbol{p}_n=J_{\gamma_n\boldsymbol{M}}\,\boldsymbol{z}_n\\
\boldsymbol{x}_{n+1}=\boldsymbol{p}_n+\gamma_n
\boldsymbol{S}(\boldsymbol{x}_n-
\boldsymbol{p}_n).
\end{array}
\right.\\[2mm]
\end{array}
\end{equation}
In turn, we derive from \cite[Theorem~2.3.1]{Siop18} that 
$(\boldsymbol{x}_n)_{n\in\NN}$ converges weakly to a point 
$\overline{\boldsymbol{x}}=
(\overline{x},\overline{\xi^*},\overline{y^*})
\in\zer\mathscr{K}_F$.
In view of Theorem~\ref{t:1}\ref{t:1v}, the proof is complete.
\end{proof}

\begin{remark}
\label{r:9}
In Proposition~\ref{p:2}, we have reduced solving the primal
problem \eqref{e:p1} and its dual \eqref{e:p2} to finding a
zero of the sum of a maximally monotone operator, a monotone
Lipschitzian operator, and a cocoercive operator via the
forward-backward-half forward splitting
method \cite{Siop18}. Let us add a few comments. 
\begin{enumerate}
\item
\label{r:9i-}
Suppose that $\ell^*=0$ and $h=0$. Then
\eqref{e:tseng} is a nonlinear 
composite extension of the monotone+skew algorithm 
\cite[Theorem~3.1]{Siop11}; the latter is obtained with 
$\phi\colon\xi\to\xi$. Another method tailored to 
inclusions involving the sum of a maximally monotone operator 
and a monotone Lipschitzian operator 
is that of \cite[Corollary~5.2]{Jmaa20}, which can also
incorporate inertial effects. Another advantage of this 
framework is that it features, through \cite[Theorem~4.8]{Jmaa20},
a strongly convergent variant which does not require any additional
assumptions on the operators. Other pertinent algorithms for this
problem are found in \cite{Cevh21,Mali20,Ryue20}.
\item
\label{r:9i}
Suppose that $g=0$, $L=0$, and $\ell^*=0$. 
Then \eqref{e:tseng} is a nonlinear composite extension of 
of the forward-backward algorithm (see, e.g., 
\cite[Corollary~28.9]{Livre1}); the latter is recovered with
$\phi\colon\xi\to\xi$.
\item
\label{r:9iii}
In the case when $h=0$ and $\ell=\iota_{\{0\}}$, a zero of 
$\mathscr{K}_F=\boldsymbol{M}+\boldsymbol{S}$ can also
be found by splitting methods which employ the resolvent of 
$\boldsymbol{S}$ given in Proposition~\ref{p:7}\ref{p:7iv} 
and do not specifically exploit its Lipschitz continuity. See for
instance \cite[Remark~2.9]{Siop11}.
\end{enumerate}
\end{remark}

A noteworthy special case of Problem~\ref{prob:1} is when
$\YY=\XX$, $L=\Id$, $h=0$, and $\ell=\iota_{\{0\}}$. In this
setting, Proposition~\ref{p:14} asserts that Problem~\ref{prob:1}
can be solved by finding a zero of the sum of the monotone
operators $\mathscr{K}_{\Psi}$ and $\boldsymbol{B}$ defined in
\eqref{e:Akt} and \eqref{e:B}, respectively. We can for instance
use the Douglas--Rachford algorithm \cite[Theorem~26.11]{Livre1}
for this task, which leads to the following implementation.

\begin{proposition}
\label{p:3}
Let $\XX$ be a real Hilbert space, let $f\in\Gamma_0(\XX)$, let
$g\in\Gamma_0(\XX)$, let $\phi\in\Gamma_0(\RR)$ be increasing and
not constant. Consider the minimization problem
\begin{equation}
\label{e:p12}
\minimize{x\in\XX}{(\phi\circ f)(x)+g(x)},
\end{equation}
under the assumption that solutions exist.
Let $z_0\in\XX$, let $\zeta_0^*\in\RR$, let $\gamma\in\RPP$, let 
$(\lambda_n)_{n\in\NN}$ be a sequence in $\left]0,2\right[$ 
such that $\sum_{n\in\NN}\lambda_n(2-\lambda_n)=\pinf$, 
define $\mm$ as in Notation~\ref{n:K}, and iterate
\begin{equation}
\label{e:DR}
\begin{array}{l}
\text{for}\;n=0,1,\ldots\\
\left\lfloor
\begin{array}{l}
\xi^*_{n}=\mm(\gamma,z_{n},\zeta_n^*)\\
x_{n}=\prox_{\gamma(\xi^*_{n}\rocky f)}{z}_{n}\\
z_{n+1}=z_n+\lambda_n\big(\prox_{\gamma g}(2x_n-z_n)-x_n\big)\\
\zeta^*_{n+1}=\zeta_n^*+\lambda_n(\xi^*_n-\zeta^*_n).
\end{array}
\right.
\end{array}
\end{equation}
Then $x_n\weakly\overline{x}$, where $\overline{x}$ is a solution
to \eqref{e:p12}.
\end{proposition} 
\begin{proof}
We first observe that \eqref{e:p12} is the 
special case of Problem~\ref{prob:1} corresponding to
$\YY=\XX$, $L=\Id$, $h=0$, and $\ell=\iota_{\{0\}}$.
It follows from Proposition~\ref{p:14}\ref{p:14iii} that 
$\zer(\mathscr{K}_{\Psi}+\boldsymbol{B})\neq\emp$, where 
$\mathscr{K}_{\Psi}$ is defined in \eqref{e:Akt} and 
$\boldsymbol{B}\colon(x,\xi^*)\mapsto\partial g(x)\times\{0\}$ 
follows from \eqref{e:B}.
Therefore, by \cite[Proposition~23.18]{Livre1}, 
$J_{\gamma\boldsymbol{B}}\colon 
(x,\xi^*)\mapsto(\prox_{\gamma g}x,\xi^*)$.
Now set, for every $n\in\NN$,
$\boldsymbol{x}_n=(x_n,\xi^*_n)\in\XX\oplus\RR$ and
$\boldsymbol{z}_n=(z_n,\zeta^*_n)\in\XX\oplus\RR$.
Then it follows from Proposition~\ref{p:71}
that \eqref{e:DR} can be written as
\begin{equation}
\label{e:DR2}
(\forall n\in\NN)\quad
\begin{array}{l}
\left\lfloor
\begin{array}{l}
\boldsymbol{x}_n=J_{\gamma 
\mathscr{K}_{\Psi}}\boldsymbol{z}_n\\
\boldsymbol{z}_{n+1}=\boldsymbol{z}_n+\lambda_n
\big(J_{\gamma 
\boldsymbol{B}}(2\boldsymbol{x}_n-\boldsymbol{z}_n)
-\boldsymbol{x}_n\big).
\end{array}
\right.
\end{array}
\end{equation}
In turn, \cite[Theorem~26.11]{Livre1} yields
$\boldsymbol{x}_n\weakly\boldsymbol{\overline{x}}$
for some $\boldsymbol{\overline{x}}\in 
\zer(\mathscr{K}_{\Psi}+\boldsymbol{B})$ and the result follows
from Proposition~\ref{p:14}\ref{p:14iii}.
\end{proof}

\begin{remark}
In the case when $\phi\colon\xi\mapsto\xi$, $\phi^*=\iota_{\{1\}}$
and therefore $\mm\equiv1$. In this context, \eqref{e:algo1} 
reduces to the algorithm of \cite[Theorem~3.1]{Svva12} and 
\eqref{e:DR} reduces to the Douglas--Rachford algorithm 
\cite[Theorem~26.11]{Livre1}.
\end{remark}

\begin{remark}\
\label{r:tseng}
\begin{enumerate}
\item 
As special cases, Examples~\ref{ex:1}--\ref{ex:2} can be solved via
Proposition~\ref{p:2} or Proposition~\ref{p:3}, depending on the 
assumptions on $L$, $\ell$, and $h$. In these settings, the
function $\mm$ in \eqref{e:algo1} and \eqref{e:DR} is computed in
Examples~\ref{ex:1b}--\ref{ex:2b}.
\item 
In the particular case of Example~\ref{ex:1}, algorithms
\eqref{e:algo1} and \eqref{e:DR} activate the inequality constraint
$f(x)\leq 0$ through the proximity operator of $f$, as described in
Example~\ref{ex:1b}. Note that general convex inequalities are hard
to treat directly in the context of standard proximal methods since
they involve the projection onto the $0$-sublevel set of $f$, which
is typically not explicit. For instance, in the case of
\eqref{e:pineq2}, the projection onto the $\ell^p$-ball is
expensive to compute \cite{Cwco21}. However, within our framework,
\eqref{e:pineq2} is solved by computing $\mm(\gamma,x,\xi^*)$ 
via Example~\ref{ex:1b} as the unique real number in $\omega\in\RP$
such that
\begin{equation}
\label{e:8y}
\omega=
\begin{cases}
0,&\text{if}\;\;\xi^*/\gamma+\|x\|_p^p\leq\eta^p;\\
\xi^*+\gamma\big(\big\|\prox_{\gamma\omega\|\cdot\|_p^p}x
\big\|_p^p-\eta^p\big),
&\text{if}\;\;\xi^*/\gamma+\|x\|_p^p>\eta^p,
\end{cases}
\end{equation}
while
\begin{equation}
\label{e:cu}
\prox_{\gamma(\omega\rocky\|\cdot\|_p^p)}x=
\begin{cases}
x,&\text{if}\;\;\xi^*/\gamma+\|x\|_p^p\leq\eta^p;\\
\big(\prox_{\gamma\omega|\cdot|^p}\xi_i\big)_{1\leq i\leq N},
&\text{if}\;\;\xi^*/\gamma+\|x\|_p^p>\eta^p,
\end{cases}
\end{equation}
where $x=(\xi_i)_{1\leq i\leq N}$ and
$\prox_{\gamma\omega|\cdot|^p}$ is made explicit in 
\cite[Example~24.38]{Livre1}.
\item
An alternative algorithmic setting that imposes no restriction 
on $L$, $\ell$, and $h$ in Problem~\ref{prob:1} will be discussed 
in Remark~\ref{r:13}. 
\end{enumerate}
\end{remark}

\section{Composite monotone inclusions}
\label{sec:5}

To study the general Problem~\ref{prob:11}, let us first 
go back to Problem~\ref{prob:1} to motivate our approach. 
Define $\mathscr{K}_{\Psi}$ as
in \eqref{e:Akt}, $B=\partial g$ and set $D=\partial\ell$,
$\boldsymbol{L}\colon\XX\oplus\RR\to\YY\colon 
(x,\xi^*)\mapsto Lx$, and 
$\boldsymbol{A}\colon\XX\oplus\RR\to 2^{\XX^*\oplus\RR}\colon 
(x,\xi^*)\mapsto(\partial h(x))\times\{0\}$. Then it follows from 
Proposition~\ref{p:14}\ref{p:14iii} that, if 
$\overline{\boldsymbol{z}}
=(\overline{x},\overline{\xi}^*)\in\XX\oplus\RR$ 
solves the monotone inclusion
\begin{equation}
\label{e:10}
\boldsymbol{0}\in
\mathscr{K}_{\Psi}\boldsymbol{\overline{z}}
+\big(\boldsymbol{L}^*\circ (B\infconv D)\circ 
\boldsymbol{L}\big)\boldsymbol{\overline{z}}+
\boldsymbol{A}\boldsymbol{\overline{z}},
\end{equation}
then $\overline{x}$ solves Problem~\ref{prob:1}. 
Here, the Kuhn--Tucker operator $\mathscr{K}_{\Psi}$ is provided by
the perturbation analysis of Proposition~\ref{p:A} and encapsulates
the nonlinear composition in Problem~\ref{prob:1}. In turn,
splitting algorithms are proposed in Section~\ref{sec:42}
to solve this inclusion.

We extend the above strategy by first defining in
Section~\ref{sec:51} adequate Kuhn--Tucker operators for
Problem~\ref{prob:11} through a perturbation analysis. In
Section~\ref{sec:52}, we exploit this analysis to connect
Problem~\ref{prob:11} to a monotone inclusion problem that will
allow us to solve it. In Section~\ref{sec:53}, we present an
algorithm to solve the parent monotone inclusion problem and,
thereby, Problem~\ref{prob:11}.

Our notational scheme is as follows.

\begin{notation}
\label{n:2}
In the setting of Problem~\ref{prob:11}, 
$\mathsf{x}=(\xi_i)_{i\in I}$ and $\mathsf{x}^*=(\xi^*_i)_{i\in I}$
are generic elements in $\RR^I$, 
$\boldsymbol{z}=(x,\mathsf{x})=(x,(\xi_i)_{i\in I})$ is a generic
element in $\ZZZ=\XX\oplus\RR^{I}$, and
$\boldsymbol{z}^*=(x^*,\mathsf{x}^*)=(x^*,(\xi^*_i)_{i\in I})$ is a
generic element in $\ZZZ^*=\XX^*\oplus\RR^{I}$.
\end{notation}

\subsection{Perturbation analysis and Kuhn--Tucker operators}
\label{sec:51}
We introduce a perturbation analysis for the nonlinear compositions
in Problem~\ref{prob:11}.

\begin{proposition}
\label{p:22}
Consider the setting of Problem~\ref{prob:11} and 
Notation~\ref{n:2}. Fix $i\in I$, let
\begin{align}
\label{e:k4}
\Psi_i\colon\XX\times\RR^I&\to\RX\nonumber\\
\big(x,(\xi_j)_{j\in I}\big)&\mapsto 
\begin{cases}
\phi_i\big(f_i(x)+\xi_i\big),&\text{if}\;\; 
f_i(x)+\xi_i\in\dom\phi_i\;\;\text{and}\;\;
(\forall j\in I\smallsetminus\{i\})\;\;\xi_j=0;\\
\pinf,&\text{if}\;\;
f_i(x)+\xi_i\notin\dom\phi_i\;\;\text{or}\;\;
(\exi j\in I\smallsetminus\{i\})\;\;\xi_j\neq 0
\end{cases}
\end{align}
be a perturbation of $\phi_i\circ f_i$, let 
$\mathscr{L}_{\Psi_i}$ be the associated Lagrangian, and let
$\mathscr{K}_{\Psi_i}$ be the associated Kuhn--Tucker operator.
Then the following hold:
\begin{enumerate}
\item 
\label{p:22i}
$\Psi_i\in\Gamma_0(\XX\oplus\RR^I)$.
\item 
\label{p:22ii}
We have 
\begin{align}
\label{e:96}
\hspace{-8mm}\mathscr{L}_{\Psi_i}\colon\ZZZ\to\RXX
\colon
\boldsymbol{z}\mapsto 
\begin{cases}
\pinf,&\text{if}\;\;x\notin \dom f_i;\\
(\xi^*_i\rocky f_i)(x)-\phi^*_i(\xi^*_i),&\text{if}\;\; x\in \dom 
f_i\,\;\text{and}\,\,\xi^*_i\in\dom\phi^*_i;\\
\minf, &\text{if}\;\;x\in \dom 
f_i\,\;\text{and}\,\,\xi^*_i\notin\dom\phi^*_i.
\end{cases}
\end{align}
\item 
\label{p:22iii}
Let $\boldsymbol{z}\in\ZZZ$ and $\boldsymbol{z}^*\in\ZZZ^*$. Then
\begin{align}
\label{e:kt5}
\boldsymbol{z}^*\in\mathscr{K}_{\Psi_i}\boldsymbol{z}
\quad\Leftrightarrow\quad 
\begin{cases}
x\in\dom f_i\;\;\text{and}\;\;\xi^*_i\in\dom\phi^*_i\\
x^*\in\partial(\xi_i^*\rocky f_i)(x)\\
\xi_i\in\partial\phi_i^*(\xi_i^*)-f_i(x)\\
(\forall j\in I\smallsetminus\{i\})\;\xi_j=0.
\end{cases}
\end{align}
\item 
\label{p:22iv}
$\mathscr{K}_{\Psi_i}$ is maximally monotone.
\end{enumerate}
\end{proposition}
\begin{proof}
\ref{p:22i}: 
Set $\boldsymbol{f}_i\colon\XX\oplus\RR\to\RX\colon 
(x,\xi)\mapsto f_i(x)+\xi$. We derive from 
\eqref{e:k4} and \eqref{e:c1} that 
\begin{equation}
\label{e:k7}
\Psi_i\colon\big(x,(\xi_j)_{j\in I}\big)\mapsto
\big(\phi_i\circ\boldsymbol{f}_i\big)(x,\xi_i)
+\sum_{j\in I\smallsetminus\{i\}}\iota_{\{0\}}(\xi_j).
\end{equation}
Since $f_i\in\Gamma_0(\XX)$, we have
$\boldsymbol{f}_i\in\Gamma_0(\XX\oplus\RR)$. 
Moreover, since $\emp\neq (\dom\phi_i)\cap f_i(\dom f_i)\subset 
(\dom\phi_i)\cap\boldsymbol{f}_i(\dom\boldsymbol{f}_i)$, 
Proposition~\ref{p:1}\ref{p:1ii} asserts that 
$\phi_i\circ\boldsymbol{f}_i\in\Gamma_0(\XX\oplus\RR)$ and the 
claim follows.

\ref{p:22ii}: 
Let $\boldsymbol{z}\in\ZZZ$.
We deduce from \eqref{e:L} that
\begin{align} 
\mathscr{L}_{\Psi_i}(\boldsymbol{z})
&=\inf_{(\xi_j)_{j\in I}\in\RR^I}
\bigg(\Psi_i\big(x,(\xi_j)_{j\in I}\big)-\sum_{j\in I}
\xi_j\xi_j^*\bigg)\nonumber\\
&=
\begin{cases}
-\sup\limits_{\xi_i\in(\dom\phi_i)-f_i(x)}
\big(\xi_i\xi_i^*-\phi_i(f_i(x)+\xi_i)\big),
&\text{if}\;\;x\in\dom f_i;\\
\pinf,&\text{if}\;\;x\notin\dom f_i\\
\end{cases}\nonumber\\
&=
\begin{cases}
\xi_i^*f_i(x)-\sup\limits_{\xi_i+
f_i(x)\in\dom\phi_i}\big(\xi_i+f_i(x)\big)
\xi_i^*-\phi_i\big(f_i(x)+\xi_i\big),
&\text{if}\;\;x\in\dom f_i;\\
\pinf,&\text{if}\;\;x\notin\dom f_i\\
\end{cases}\nonumber\\
&=
\begin{cases}
\xi_i^*f_i(x)-\phi_i^*(\xi_i^*),&\text{if}\;\; x\in\dom f_i;\\
\pinf,&\text{if}\;\; x\notin \dom f_i.
\end{cases}
\end{align}
The claim therefore follows from 
Lemma~\ref{l:3}\ref{l:3ii} and \eqref{e:66}.

\ref{p:22iii}: 
Note that, if $x\notin\dom f_i$, \ref{p:22ii} yields
$\mathscr{L}_{\Psi_i}(\boldsymbol{z})=\pinf$ and, hence, 
$\mathscr{K}_{\Psi_i}(\boldsymbol{z})=\emp$ 
in view of \eqref{e:K}.
Similarly, if $\xi_i^*\notin\dom\phi_i^*$, then 
$-\mathscr{L}_{\Psi_i}(\boldsymbol{z})=\pinf$, which yields 
$\mathscr{K}_{\Psi_i}(\boldsymbol{z})=\emp$.
Now suppose that $x\in\dom f_i$ and $\xi_i^*\in\dom\phi_i^*$.
Arguing along the same lines as in the proof of 
Proposition~\ref{p:A}\ref{p:Aii}, we deduce that 
$\partial(\mathscr{L}_{\Psi_i}(\cdot,(\xi^*_j)_{j\in I}))
=\partial(\xi^*_i\rocky f_i)$ and that 
$(\xi_j)_{j\in I}\in\partial
\big(-\mathscr{L}_{\Psi_i}(x,\cdot)\big)(\xi^*_j)_{j\in I}$ if and
only if 
$\xi_i\in\partial\phi_i^*(\xi_i^*)-f_i(x)$ and, for every 
$j\in I\smallsetminus\{i\}$, $\xi_j=0$.
Altogether, \eqref{e:K} yields \eqref{e:kt5}. 

\ref{p:22iv}: This follows from \ref{p:22i} and
Lemma~\ref{l:9}\ref{l:9i}.
\end{proof}

\subsection{Monotone inclusion formulation}
\label{sec:52}
We associate Problem~\ref{prob:11} with a monotone inclusion
problem.

\begin{proposition} 
\label{p:51}
Consider the setting of Problem~\ref{prob:11} and
Notation~\ref{n:2}. Define the operators
$(\mathscr{K}_{\Psi_i})_{i\in I}$ as in \eqref{e:kt5}, and set
\begin{equation}
\label{e:Au}
\begin{cases}
\boldsymbol{A}\colon\ZZZ\to 2^{\ZZZ^*}\colon 
\boldsymbol{z}\mapsto Ax\times\{0\}\\
\boldsymbol{C}\colon\ZZZ\to\ZZZ^*\colon 
\boldsymbol{z}\mapsto(Cx,0)\\
\boldsymbol{Q}\colon\ZZZ\to \ZZZ^*\colon 
\boldsymbol{z}\mapsto(Qx,0)\\
(\forall k\in K)\;\;\boldsymbol{L}_k\colon\ZZZ\to\YY_k\colon
\boldsymbol{z}\mapsto L_kx.
\end{cases}
\end{equation}
Consider the inclusion problem 
\begin{equation}
\label{e:11}
\text{find}\;\;\boldsymbol{z}\in\ZZZ\;\;\text{such that}\;\;
\boldsymbol{0}\in\sum_{i\in I}
\mathscr{K}_{\Psi_i}\boldsymbol{z}
+\sum_{k\in K}\big(\boldsymbol{L}^*_k\circ (B_k\infconv D_k)\circ 
\boldsymbol{L}_k\big)\boldsymbol{{z}}+
\boldsymbol{A}\boldsymbol{{z}}+
\boldsymbol{C}\boldsymbol{{z}}+
\boldsymbol{Q}\boldsymbol{{z}}.
\end{equation}
Then the following hold:
\begin{enumerate}
\item
\label{p:51i}
Suppose that 
$\overline{\boldsymbol{z}}=
(\overline{x},({\overline{\xi^*_i}})_{i\in I})\in\ZZZ$ solves 
\eqref{e:11}. Then $\overline{x}$ solves Problem~\ref{prob:11}.
\item
\label{p:51ii}
Suppose that $\overline{x}\in\XX$ solves Problem~\ref{prob:11} and
that, for every $i\in I$, 
$(\inte\dom\phi_i)\cap f_i(\dom f_i)\neq\emp$. 
Then there exists $(\overline{\xi^*_i})_{i\in I}\in\RR^I$ such 
that $(\overline{x},(\overline{\xi^*_i})_{i\in I})$ solves
\eqref{e:11}.
\end{enumerate}
\end{proposition}
\begin{proof}
In view of \eqref{e:kt5} and \eqref{e:Au}, 
for every $\overline{\boldsymbol{z}}=
(\overline{x},({\overline{\xi^*_i}})_{i\in I})\in\ZZZ$,
\begin{eqnarray}
\label{e:24}
\boldsymbol{0}&\in&\sum_{i\in I}
\mathscr{K}_{\Psi_i}\overline{\boldsymbol{z}}
+\sum_{k\in K}\big(\boldsymbol{L}^*_k\circ (B_k\infconv D_k)\circ 
\boldsymbol{L}_k\big)\overline{\boldsymbol{z}}+
\boldsymbol{A}\overline{\boldsymbol{z}}+
\boldsymbol{C}\overline{\boldsymbol{z}}+
\boldsymbol{Q}\overline{\boldsymbol{z}}\nonumber\\
&\Leftrightarrow&
\begin{cases}
0\in\Sum_{i\in I}\partial
(\overline{\xi_i^*}\rocky f_i)(\overline{x})
+\Sum_{k\in K}\big(L^*_k\circ (B_k\infconv D_k)\circ L_k\big)
\overline{x}+A\overline{x}+C\overline{x}+Q\overline{x}\\
(\forall i\in I)\quad 0\in\partial\phi_i^*(\overline{\xi_i^*})
-f_i(\overline{x}),
\end{cases}\nonumber\\
&\Leftrightarrow&
\begin{cases}
0\in\Sum_{i\in I}\partial
(\overline{\xi_i^*}\rocky f_i)(\overline{x})
+\Sum_{k\in K}\big(L^*_k\circ (B_k\infconv D_k)\circ L_k\big)
\overline{x}+A\overline{x}+C\overline{x}+Q\overline{x}\\
(\forall i\in I)\quad \overline{\xi_i^*}\in\partial\phi_i\big(
f_i(\overline{x})\big).
\end{cases}
\end{eqnarray}

\ref{p:51i}:
This follows from \eqref{e:24} and 
Proposition~\ref{p:1}\ref{p:1iii-}.

\ref{p:51ii}: 
This follows from Proposition~\ref{p:1}\ref{p:1iii} and 
\eqref{e:24}. 
\end{proof}

\subsection{Block-iterative algorithm}
\label{sec:53}
We provide a flexible splitting algorithm for solving
Problem~\ref{prob:11}, in which all the operators and functions
present in \eqref{e:prob11} are activated separately. The proposed
algorithm allows for block-iterative implementations in the sense
that, at a given iteration, only a subgroup of functions and
operators needs to be employed, as opposed to all of them as in
classical methods. This feature is particularly
important in large-scale applications, where it is either
impossible or inefficient to evaluate each the operators at each
iteration. This algorithm is an offspring of that of \cite{Moor22},
where one will find more details on its construction. In a
nutshell, its principle is to recast the parent inclusion problem
\eqref{e:11} as that
of finding a zero of a saddle operator that acts on a bigger space.
This task is achieved by successive projections onto half-spaces
containing the set of zeros of the saddle operator. At a given
iteration, the construction of such a half-space involves the
resolvents of the selected operators.

The following vector version of Notation~\ref{n:K} will be needed.
\begin{notation}
\label{n:K2} 
Let $\gamma\in\RPP$, let $x\in\XX$, and let $\mathsf{x}\in\RR^{I}$.
For every $i\in I$, let $\mm_i$ be the function defined as in
Notation~\ref{n:K} with respect to $(\phi_i,f_i)$, let
$\mathsf{e}_i\in\RR^{I}$ be the $i$th canonical vector, let
$[\mathsf{x}]_i$ be the $i$th coordinate of $\mathsf{x}$, and set
$\mathsf{w}_i(\gamma,x,\mathsf{x})=\mathsf{x}+
(\mm_i(\gamma,x,[\mathsf{x}]_i)-[\mathsf{x}]_i)\mathsf{e}_i$.
\end{notation}

\begin{theorem}
\label{t:3} 
Consider the setting of Problem~\ref{prob:11} and of
Notations~\ref{n:1}, \ref{n:2}, and~\ref{n:K2}. Suppose that 
$\XX$ and $(\YY_k)_{k\in K}$ are real Hilbert spaces with scalar
products denoted by $\scal{\cdot}{\cdot}$ and that
Problem~\ref{prob:11} admits at least one solution. In
addition, let $\sigma\in\left]1/(4\beta),\pinf\right[$, let
$\varepsilon\in\left]0,\min\{1,1/(\chi+\sigma)\}\right[$,
let $(\gamma_n)_{n\in\NN}$ and $(\lambda_n)_{n\in\NN}$ be 
sequences in $\left[\varepsilon,1/(\chi+\sigma)\right]$ and 
$\left[\varepsilon,2-\varepsilon\right]$, respectively, and,
for every $j\in I\cup K$, let $(\mu_{j,n})_{n\in\NN}$, 
$(\nu_{j,n})_{n\in\NN}$, and $(\sigma_{j,n})_{n\in\NN}$ 
be sequences in $\left[\varepsilon,1/\sigma\right]$,
$\left[\varepsilon,1/\sigma\right]$, and 
$\left[\varepsilon,1/\varepsilon\right]$,
respectively. Moreover, let $x_0\in\XX$, let 
$\mathsf{x}_0\in\RR^{I}$, for every $i\in I$, let
$\{{y}_{i,0},{u}_{i,0},{v}_{i,0}^*\}\subset\XX$ and
$\{\mathsf{y}_{i,0},\mathsf{u}_{i,0},\mathsf{v}_{i,0}^*\}
\subset\RR^{I}$, and, for every $k\in K$, let
$\{y_{k,0},u_{k,0},v_{k,0}^*\}\subset\YY_k$. Further, let 
$(P_n)_{n\in\NN}$ be nonempty subsets of $I\cup K$ such that
\begin{equation}
\label{e:23931}
P_0=I\cup K\quad\text{and}\quad(\exi M\in\NN)(\forall n\in\NN)\;\;
\bigcup_{k=n}^{n+M}P_k=I\cup K.
\end{equation}
Iterate
\begin{align}
\label{e:long17}
\hspace{-9mm}
&
\begin{array}{l}
\text{for}\;n=0,1,\ldots\\
\left\lfloor        
\begin{array}{l}
l_n^*=Qx_n+\sum_{i\in I}v_{i,n}^*
+\sum_{k\in K}L_k^*v_{k,n}^*;\;
\mathsf{\ell}_n^*=\sum_{i\in I}\mathsf{v}_{i,n}^*;\\
a_n=J_{\gamma_n{A}}\big(x_n-\gamma_n(l_n^*+Cx_n)\big);\;
\mathsf{a}_n=\mathsf{x}_n-\gamma_n\ell_n^*;\\
a_n^*=\gamma_n^{-1}(x_n-a_n)-l_n^*+Qa_n;\;
\xi_n=\|a_n-x_n\|^2+\|\mathsf{a}_n-\mathsf{x}_n\|^2;\\[1mm]
\text{for every}\;i\in I\cap P_n\\
\left\lfloor
\begin{array}{l}
\mathsf{b}_{i,n}=\mathsf{w}_i(\mu_{i,n},y_{i,n}
\!+\!\mu_{i,n}v_{i,n}^*,\mathsf{y}_{i,n}
\!+\!\mu_{i,n}\mathsf{v}_{i,n}^*);\;
b_{i,n}=\prox_{\mu_{i,n}([\mathsf{b}_{i,n}]_i\rocky f_i)}
\big(y_{i,n}\!+\!\mu_{i,n}v_{i,n}^*\big);\\
e_{i,n}=b_{i,n}-a_n;\;
\mathsf{e}_{i,n}=\mathsf{b}_{i,n}-\mathsf{a}_n;\\
e_{i,n}^*=\sigma_{i,n}(x_n -y_{i,n}-u_{i,n})+v_{i,n}^*;\;
\mathsf{e}_{i,n}^*=\sigma_{i,n}(\mathsf{x}_n
-\mathsf{y}_{i,n}-\mathsf{u}_{i,n})+\mathsf{v}_{i,n}^*;\\
q_{i,n}^*=\mu_{i,n}^{-1}(y_{i,n}-b_{i,n})+v_{i,n}^*-e_{i,n}^*;\;
\mathsf{q}_{i,n}^*=
\mu_{i,n}^{-1}(\mathsf{y}_{i,n}-\mathsf{b}_{i,n})
+\mathsf{v}_{i,n}^*-\mathsf{e}_{i,n}^*;\\
t_{i,n}^*=\nu_{i,n}^{-1}u_{i,n}+v_{i,n}^*-e_{i,n}^*;\;
\mathsf{t}_{i,n}^*=\nu_{i,n}^{-1}\mathsf{u}_{i,n}
+\mathsf{v}_{i,n}^*-\mathsf{e}_{i,n}^*;\\
\eta_{i,n}=\|b_{i,n}-y_{i,n}\|^2+
\|\mathsf{b}_{i,n}-\mathsf{y}_{i,n}\|^2
+\|u_{i,n}\|^2+\|\mathsf{u}_{i,n}\|^2;
\end{array}
\right.\\[14mm]
\text{for every}\;k\in K\cap P_n\\
\left\lfloor
\begin{array}{l}
b_{k,n}=J_{\mu_{k,n}B_k}\big(y_{k,n}+\mu_{k,n}v_{k,n}^*\big);\;
d_{k,n}=J_{\nu_{k,n}D_k}\big(u_{k,n}+\nu_{k,n}v_{k,n}^*\big);\\
e_{k,n}=b_{k,n}+d_{k,n}-L_ka_n;\;
e_{k,n}^*=\sigma_{k,n}\big(L_kx_n-y_{k,n}-u_{k,n}\big)+v_{k,n}^*;\\
q_{k,n}^*=\mu_{k,n}^{-1}(y_{k,n}-b_{k,n})+v_{k,n}^*-e_{k,n}^*;\;
t_{k,n}^*=\nu_{k,n}^{-1}(u_{k,n}-d_{k,n})+v_{k,n}^*-e_{k,n}^*;\\
\eta_{k,n}=\|b_{k,n}-y_{k,n}\|^2+\|d_{k,n}-u_{k,n}\|^2;
\end{array}
\right.\\[4mm]
\text{for every}\;i\in I\smallsetminus P_n\\
\left\lfloor
\begin{array}{l}
b_{i,n}=b_{i,n-1};\:
\mathsf{b}_{i,n}=\mathsf{b}_{i,n-1};\:
e_{i,n}=b_{i,n}-a_n;\:
\mathsf{e}_{i,n}=\mathsf{b}_{i,n}-\mathsf{a}_n;
e_{i,n}^*=e_{i,n-1}^*;\\
\mathsf{e}_{i,n}^*=\mathsf{e}_{i,n-1}^*;\:
q_{i,n}^*=q_{i,n-1}^*;\:
\mathsf{q}_{i,n}^*=\mathsf{q}_{i,n-1}^*;\:
t_{i,n}^*=t_{i,n-1}^*;\:
\mathsf{t}_{i,n}^*=\mathsf{t}_{i,n-1}^*;\:
\eta_{i,n}=\eta_{i,n-1};
\end{array}
\right.\\[4mm]
\text{for every}\;k\in K\smallsetminus P_n\\
\left\lfloor
\begin{array}{l}
b_{k,n}=b_{k,n-1};\;
d_{k,n}=d_{k,n-1};\;
e_{k,n}=b_{k,n}+d_{k,n}-L_ka_n;\;
e_{k,n}^*=e_{k,n-1}^*;\\
q_{k,n}^*=q_{k,n-1}^*;\;
t_{k,n}^*=t_{k,n-1}^*;\;
\eta_{k,n}=\eta_{k,n-1};
\end{array}
\right.\\[4mm]
\begin{aligned}
p_n^*&=a_n^*+\sum_{i\in I}e_{i,n}^*+\sum_{k\in K}L_k^*e_{k,n}^*;\;
\mathsf{p}_n^*=\sum_{i\in I}\mathsf{e}_{i,n}^*;\\
\Delta_n
&=-(4\beta)^{-1}\big(\xi_n
+\textstyle\sum_{j\in I\cup K}\eta_{j,n}\big)+
\scal{x_n-a_n}{p_n^*}+
\scal{\mathsf{x}_n-\mathsf{a}_n}{\mathsf{p}_n^*}\\
&\textstyle
\quad\;+\textstyle\sum_{i\in I}
\big(\scal{y_{i,n}-b_{i,n}}{q_{i,n}^*}
+\scal{u_{i,n}}{t_{i,n}^*}
+\scal{e_{i,n}}{v_{i,n}^*-e_{i,n}^*}\big)\\
&\quad\;+\textstyle\sum_{i\in I}
\big(\scal{\mathsf{y}_{i,n}-\mathsf{b}_{i,n}}{\mathsf{q}_{i,n}^*}
+\scal{\mathsf{u}_{i,n}}{\mathsf{t}_{i,n}^*}
+\scal{\mathsf{e}_{i,n}}{\mathsf{v}_{i,n}^*-\mathsf{e}_{i,n}^*}
\big)\\
&\textstyle
\quad\;+\textstyle\sum_{k\in K}
\big(\scal{y_{k,n}-b_{k,n}}{q_{k,n}^*}
+\scal{u_{k,n}-d_{k,n}}{t_{k,n}^*}
+\scal{e_{k,n}}{v_{k,n}^*-e_{k,n}^*}\big);
\end{aligned}\\
\text{if}\;\Delta_n>0\\
\left\lfloor
\begin{array}{l}
\alpha_n\!=\!\|p_n^*\|^2\!+\!\|\mathsf{p}_n^*\|^2\!+\!\!\!
\Sum_{j\in I\cup K}\!\!\!
\|q_{j,n}^*\|^2\!+\!\|t_{j,n}^*\|^2\!+\!\|e_{j,n}\|^2
\!+\!\Sum_{i\in I}\|\mathsf{q}_{i,n}^*\|^2\!+\!
\|\mathsf{t}_{i,n}^*\|^2
\!+\!\|\mathsf{e}_{i,n}\|^2;\\
\theta_n=\lambda_n\Delta_n/\alpha_n;\\
x_{n+1}=x_n-\theta_n p_n^*;\;
\mathsf{x}_{n+1}=\mathsf{x}_n-\theta_n\mathsf{p}_n^*;\\
\text{for every}\;j\in I\cup K\\
\left\lfloor
\begin{array}{l}
y_{j,n+1}=y_{j,n}-\theta_nq_{j,n}^*;\;
u_{j,n+1}=u_{j,n}-\theta_nt_{j,n}^*;\;
v_{j,n+1}^*=v_{j,n}^*-\theta_ne_{j,n};
\end{array}
\right.\\[2mm]
\text{for every}\;i\in I\\
\left\lfloor
\begin{array}{l}
\mathsf{y}_{i,n+1}=\mathsf{y}_{i,n}-\theta_n\mathsf{q}_{i,n}^*;\;
\mathsf{u}_{i,n+1}=\mathsf{u}_{i,n}-\theta_n\mathsf{t}_{i,n}^*;\;
\mathsf{v}_{i,n+1}^*=\mathsf{v}_{i,n}^*-\theta_n\mathsf{e}_{i,n};
\end{array}
\right.\\[2mm]
\end{array}
\right.\\
\text{else}\\
\left\lfloor
\begin{array}{l}
x_{n+1}=x_n;\;
\mathsf{x}_{n+1}=\mathsf{x}_n;\\
\text{for every}\;j\in I\cup K\\
\left\lfloor
\begin{array}{l}
y_{j,n+1}=y_{j,n};\;
u_{j,n+1}=u_{j,n};\;
v_{j,n+1}^*=v_{j,n}^*;\;
\end{array}
\right.\\[2mm]
\text{for every}\;i\in I\\
\left\lfloor
\begin{array}{l}
\mathsf{y}_{i,n+1}=\mathsf{y}_{i,n};\;
\mathsf{u}_{i,n+1}=\mathsf{u}_{i,n};\;
\mathsf{v}_{i,n+1}^*=\mathsf{v}_{i,n}^*.
\end{array}
\right.\\
\end{array}
\right.\\[9mm]
\end{array}
\right.
\end{array}
\end{align}
Then $x_n\weakly\overline{x}$, where $\overline{x}$ is a solution
to Problem~\ref{prob:11}. 
\end{theorem}
\begin{proof}
Set $P=I\cup K$, let $(\mathscr{K}_{\Psi_i})_{i\in I}$, 
$(\boldsymbol{L}_k)_{k\in K}$, $\boldsymbol{A}$, $\boldsymbol{C}$,
and $\boldsymbol{Q}$ be the operators defined in 
\eqref{e:kt5} and $\eqref{e:Au}$,
and, for every $j\in P$, define
\begin{equation}
\label{e:w1}
\boldsymbol{B}_j=
\begin{cases}
\mathscr{K}_{\Psi_j},&\text{if}\;j\in I;\\
B_j,&\text{if}\;j\in K,
\end{cases}\quad 
\boldsymbol{D}_j=
\begin{cases}
N_{\{0\}},&\text{if}\;j\in I;\\
D_j,&\text{if}\;j\in K,
\end{cases}\quad\text{and}\quad 
\boldsymbol{M}_j=
\begin{cases}
\boldsymbol{\Id},&\text{if}\;j\in I;\\
\boldsymbol{L}_j,&\text{if}\;j\in K.
\end{cases}
\end{equation}
Then $\boldsymbol{A}$ is maximally monotone, $\boldsymbol{C}$ is
$\beta$-cocoercive, $\boldsymbol{Q}$ is monotone and
$\chi$-Lipschitzian, and, for every $j\in P$, $\boldsymbol{B}_j$
and $\boldsymbol{D}_j$ are maximally monotone (see
Proposition~\ref{p:22}\ref{p:22iv}) and
$\boldsymbol{M}_j$ is linear and bounded. It follows from
Proposition~\ref{p:51}\ref{p:51ii} and \eqref{e:w1} that there 
exists $\overline{\boldsymbol{z}}=
(\overline{x},(\overline{\xi_i^*})_{i\in I})\in\XX\oplus\RR^{I}$
such that
\begin{equation}
\label{e:1p}
\boldsymbol{0}\in\Sum_{j\in P}
\big(\boldsymbol{M}_j^*\circ
(\boldsymbol{B}_j\infconv 
\boldsymbol{D}_j)\circ\boldsymbol{M}_j\big)
\overline{\boldsymbol{z}}+\boldsymbol{A}\overline{\boldsymbol{z}}
+\boldsymbol{C}\overline{\boldsymbol{z}}+
\boldsymbol{Q}\overline{\boldsymbol{z}}.
\end{equation}
Moreover, we derive from \eqref{e:kt5} and 
Proposition~\ref{p:71} that 
\begin{equation}
\label{e:w2}
(\forall i\in I)(\forall\gamma\in\RPP)\quad
J_{\gamma\mathscr{K}_{\Psi_i}}\colon (x,\mathsf{x})\mapsto 
\big(\prox_{\gamma(\mm_i(\gamma,x,[\mathsf{x}]_i)\rocky f_i)}x,
\mathsf{w}_i(\gamma,x,\mathsf{x})\big).
\end{equation}
Now, for every $n\in\NN$, set
\begin{equation}
\label{e:w3}
\boldsymbol{z}_n=(x_n,\mathsf{x}_n),\;
\boldsymbol{l}_n^*=(l_n^*,\ell_n^*),\;
\boldsymbol{a}_n=(a_n,\mathsf{a}_n),\;
\boldsymbol{a}_n^*=(a_n^*,\mathsf{0}),\;
\boldsymbol{p}_n^*=(p_n^*,\mathsf{p}_n^*),
\end{equation}
for every $j\in I$,  set
\begin{equation}
\label{e:w4}
\begin{cases}
\boldsymbol{b}_{j,n}=(b_{j,n},\mathsf{b}_{j,n}),\;
\boldsymbol{e}_{j,n}^*=(e_{j,n}^*,\mathsf{e}_{j,n}^*),\;
\boldsymbol{e}_{j,n}=(e_{j,n},\mathsf{e}_{j,n}),\;
\boldsymbol{q}_{j,n}^*=(q_{j,n}^*,\mathsf{q}_{j,n}^*),\;
\boldsymbol{t}_{j,n}^*=(t_{j,n}^*,\mathsf{t}_{j,n}^*),\\
\boldsymbol{y}_{j,n}=(y_{j,n},\mathsf{y}_{j,n}),\;
\boldsymbol{u}_{j,n}=(u_{j,n},\mathsf{u}_{j,n}),\;
\boldsymbol{v}_{j,n}^*=(v_{j,n}^*,\mathsf{v}_{j,n}^*),\;
\boldsymbol{d}_{j,n}=(0,0),
\end{cases}
\end{equation}  
and, for every $j\in K$, set 
\begin{equation}
\label{e:w5}
\begin{cases}
\boldsymbol{b}_{j,n}=b_{j,n},\;
\boldsymbol{e}_{j,n}^*=e_{j,n}^*,\;
\boldsymbol{e}_{j,n}=e_{j,n},\;
\boldsymbol{q}_{j,n}^*=q_{j,n}^*,\;
\boldsymbol{t}_{j,n}^*=t_{j,n}^*,\\
\boldsymbol{y}_{j,n}=y_{j,n},\;
\boldsymbol{u}_{j,n}=u_{j,n},\;
\boldsymbol{v}_{j,n}^*=v_{j,n}^*,\;
\boldsymbol{d}_{j,n}=d_{j,n}.
\end{cases}
\end{equation}
Using \eqref{e:w1}, \eqref{e:w2}, \eqref{e:w3}, \eqref{e:w4},
\eqref{e:w5}, and elementary manipulations, we reduce
\eqref{e:long17} to 
\begin{align}
\label{e:long19}
&
\begin{array}{l}
\text{for}\;n=0,1,\ldots\\
\left|
\begin{array}{l}
\boldsymbol{l}_n^*=\boldsymbol{Q}\boldsymbol{z}_n+
\sum_{j\in P}\boldsymbol{M}_j^*\boldsymbol{v}_{j,n}^*;\\
\boldsymbol{a}_n=J_{\gamma_n\boldsymbol{A}}\big(\boldsymbol{z}_n
-\gamma_n(\boldsymbol{l}_n^*+
\boldsymbol{C}\boldsymbol{z}_n)\big);\\
\boldsymbol{a}_n^*=\gamma_n^{-1}(\boldsymbol{z}_n
-\boldsymbol{a}_n)-\boldsymbol{l}_n^*+
\boldsymbol{Q}\boldsymbol{a}_n;\\
\xi_n=\|\boldsymbol{a}_n-\boldsymbol{z}_n\|^2;\\
\text{for every}\;j\in P_n\\
\left\lfloor
\begin{array}{l}
\boldsymbol{b}_{j,n}=J_{\mu_{j,n}\boldsymbol{B}_j}
\big(\boldsymbol{y}_{j,n}+\mu_{j,n}\boldsymbol{v}_{j,n}^*\big);\\
\boldsymbol{d}_{j,n}=J_{\nu_{j,n}\boldsymbol{D}_j}
\big(\boldsymbol{u}_{j,n}+\nu_{j,n}\boldsymbol{v}_{j,n}^*\big);\\
\boldsymbol{e}_{j,n}=\boldsymbol{b}_{j,n}+\boldsymbol{d}_{j,n}
-{\boldsymbol{M}}_j\boldsymbol{a}_n;\\
\boldsymbol{e}_{j,n}^*=\sigma_{j,n}
\big(\boldsymbol{M}_j\boldsymbol{z}_n
-\boldsymbol{y}_{j,n}-\boldsymbol{u}_{j,n}\big)+
\boldsymbol{v}_{j,n}^*;\\
\boldsymbol{q}_{j,n}^*=\mu_{j,n}^{-1}
(\boldsymbol{y}_{j,n}-\boldsymbol{b}_{j,n})+\boldsymbol{v}_{j,n}^*
-\boldsymbol{e}_{j,n}^*;\\
\boldsymbol{t}_{j,n}^*=\nu_{j,n}^{-1}(\boldsymbol{u}_{j,n}
-\boldsymbol{d}_{j,n})+\boldsymbol{v}_{j,n}^*-
\boldsymbol{e}_{j,n}^*;\\
\eta_{j,n}=\|\boldsymbol{b}_{j,n}-\boldsymbol{y}_{j,n}\|^2
+\|\boldsymbol{d}_{j,n}-\boldsymbol{u}_{j,n}\|^2;
\end{array}
\right.\\
\end{array}      
\right.
\end{array}
\\
&\nonumber
\begin{array}{l}
\left\lfloor
\begin{array}{l} 
\text{for every}\;j\in P\smallsetminus P_n\\
\left\lfloor
\begin{array}{l}
\boldsymbol{b}_{j,n}=\boldsymbol{b}_{j,n-1};\;
\boldsymbol{d}_{j,n}=\boldsymbol{d}_{j,n-1};\;
\boldsymbol{e}_{j,n}^*=\boldsymbol{e}_{j,n-1}^*;\;
\boldsymbol{q}_{j,n}^*=\boldsymbol{q}_{j,n-1}^*;\;
\boldsymbol{t}_{j,n}^*=\boldsymbol{t}_{j,n-1}^*;\\
\boldsymbol{e}_{j,n}=\boldsymbol{b}_{j,n}+\boldsymbol{d}_{j,n}
-\boldsymbol{M}_j\boldsymbol{a}_n;\;
\eta_{j,n}=\eta_{j,n-1};\\
\end{array}
\right.\\[4mm]
\boldsymbol{p}_n^*=\boldsymbol{a}_n^*+
\sum_{j\in P}{\boldsymbol{M}}_j^*\boldsymbol{e}_{j,n}^*;\\[2mm]
\begin{aligned}
\Delta_n&=-(4\beta)^{-1}\big(\xi_n
+\textstyle{\sum_{j\in P}}\eta_{j,n}\big)+\scal{\boldsymbol{z}_n-
\boldsymbol{a}_n}{\boldsymbol{p}_n^*}\\[1mm]
&\textstyle
\quad\;+\sum_{j\in P}\big(\scal{\boldsymbol{y}_{j,n}-
\boldsymbol{b}_{j,n}}{\boldsymbol{q}_{j,n}^*}
+\scal{\boldsymbol{u}_{j,n}-\boldsymbol{d}_{j,n}}
{\boldsymbol{t}_{j,n}^*}+\scal{\boldsymbol{e}_{j,n}}
{\boldsymbol{v}_{j,n}^*-\boldsymbol{e}_{j,n}^*}\big);
\end{aligned}\\
\text{if}\;\Delta_n>0\\
\left\lfloor
\begin{array}{l}
\theta_n=\lambda_n\Delta_n/
\big(\|\boldsymbol{p}_n^*\|^2+\sum_{j\in P}\big(
\|\boldsymbol{q}_{j,n}^*\|^2+\|\boldsymbol{t}_{j,n}^*\|^2
+\|\boldsymbol{e}_{j,n}\|^2\big)\big);\\
\boldsymbol{z}_{n+1}=\boldsymbol{z}_n-\theta_n\boldsymbol{p}_n^*;\\
\text{for every}\;j\in P\\
\left\lfloor
\begin{array}{l}
\boldsymbol{y}_{j,n+1}=
\boldsymbol{y}_{j,n}-\theta_n\boldsymbol{q}_{j,n}^*;\;
\boldsymbol{u}_{j,n+1}=
\boldsymbol{u}_{j,n}-\theta_n\boldsymbol{t}_{j,n}^*;\;
\boldsymbol{v}_{j,n+1}^*=
\boldsymbol{v}_{j,n}^*-\theta_n\boldsymbol{e}_{j,n};
\end{array}
\right.\\[1mm]
\end{array}
\right.\\
\text{else}\\
\left\lfloor
\begin{array}{l}
\boldsymbol{z}_{n+1}=\boldsymbol{z}_n;\\
\text{for every}\;j\in P\\
\left\lfloor
\begin{array}{l}
\boldsymbol{y}_{j,n+1}=\boldsymbol{y}_{j,n};\;
\boldsymbol{u}_{j,n+1}=\boldsymbol{u}_{j,n};\;
\boldsymbol{v}_{j,n+1}^*=\boldsymbol{v}_{j,n}^*.
\end{array}
\right.\\[1mm]
\end{array}
\right.\\[9mm]
\end{array}
\right.
\end{array}
\end{align}
Since \eqref{e:1p} corresponds to 
a particular instance of \cite[Problem~1]{Moor22} and 
\eqref{e:long19} is the application of \cite[Algorithm~1]{Moor22} 
to solve \eqref{e:1p}, the result follows from 
\cite[Theorem~1(iv)]{Moor22}.
\end{proof}

\begin{remark}\
\label{r:5}
\begin{enumerate}
\item
The weakly convergent algorithm \eqref{e:long17} is based on
\cite[Algorithm~1]{Moor22}. We can also derive from 
\cite[Theorem~2(iv)]{Moor22} a strongly convergent variant of it
which does not require additional assumptions on
Problem~\ref{prob:11}. 
\item
Although this point is omitted for simplicity, 
\cite[Algorithm~1]{Moor22} and \cite[Theorem~2(iv)]{Moor22} allow
for asynchronous implementations, in the sense that any iteration
can incorporate the result of calculations initiated at earlier
iterations. In turn, \eqref{e:long17} (and its strongly convergent
counterpart) can also be executed asynchronously.
\end{enumerate}
\end{remark}

\section{Applications}
\label{sec:6}

We describe a few applications of our framework to variational
problems.

\begin{example}
\label{ex:11}
In Problem~\ref{prob:11}, suppose that, for every $k\in K$,
$B_k=N_{E_k}$ and $D_k=N_{F_k}$, where $E_k$ and $F_k$ are nonempty
closed convex subsets of $\YY_k$ such that
$0\in\sri(\barc E_k-\barc F_k)$. Define
\begin{equation}
\label{e:v3}
G=\menge{x\in\XX}{(\forall k\in K)\;L_kx\in E_k+F_k}.
\end{equation}
Note that, for every $k\in K$, 
$0\in\sri(\barc E_k-\barc F_k)=
\sri(\dom\iota_{E_k}^*-\dom\iota_{F_k}^*)$ and, therefore,
Lemma~\ref{l:4} and \eqref{e:v3} imply that 
\begin{align}
\sum_{k\in K}L_k^*\circ(N_{E_k}\infconv N_{F_k})\circ L_k
&=\sum_{k\in K}L_k^*\circ(\partial\iota_{E_k}
\infconv\partial\iota_{F_k})\circ L_k
\nonumber\\
&=\sum_{k\in K}L_k^*\circ\partial(\iota_{E_k}
\infconv\iota_{F_k})\circ L_k\nonumber\\
&=\sum_{k\in K}L_k^*\circ\partial(\iota_{E_k+F_k})\circ L_k
\nonumber\\
&\subset\partial\bigg(\sum_{k\in K}\iota_{E_k+F_k}\,\circ L_k\bigg)
\nonumber\\
&=\partial\bigg(\sum_{k\in K}\iota_{L_k^{-1}(E_k+F_k)}\bigg)
\nonumber\\
&=\partial\iota_{\bigcap_{k\in K}L_k^{-1}(E_k+F_k)}
\nonumber\\
&=N_G.
\end{align}
Thus, every solution to Problem~\ref{prob:11} solves the 
variational inequality problem
\begin{equation}
\label{e:u7}
\text{find}\;\;x\in G\;\;\text{such that}\quad
(\forall y\in G)\bigg(\exi x^*\in Ax+\sum_{i\in I}
\partial (\phi_i\circ f_i)(x)\bigg)\;\;\pair{y-x}{x^*+Cx+Qx}\geq 0.
\end{equation}
The type of constraint set \eqref{e:v3} appears in
\cite{Moor22,Wang20}. To solve \eqref{e:u7} via \eqref{e:long17},
we require the resolvents $J_{\mu_{k,n}B_k}=\proj_{E_k}$, and
$J_{\nu_{k,n}D_k}=\proj_{F_k}$.
\end{example}

\begin{example}
\label{ex:32}
Consider the implementation of Problem~\ref{prob:11} in which
\begin{equation}
A=\partial h_1,\;\;C=\nabla h_2,\;\;Q=0,
\;\;\text{and}\;\;(\forall k\in K)\;\;
B_k=\partial g_k\;\,\text{and}\;\,D_k=\partial\ell_k,
\end{equation}
where $h_1\in\Gamma_0(\XX)$, $h_2\colon\XX\to\RR$ is a 
differentiable convex function with a Lipschitzian gradient, and 
for every $k\in K$, $g_k$ and $\ell_k$ are functions in 
$\Gamma_0(\YY_k)$ such that $0\in\sri(\dom g_k^*-\dom\ell_k^*)$.
Then, if $\overline{x}\in\XX$ solves Problem~\ref{prob:11}, it
solves 
\begin{equation}
\label{e:p3}
\minimize{x\in\XX}{\sum_{i\in I}(\phi_i\circ f_i)(x)+
\sum_{k\in K}(g_k\infconv\ell_k)(L_kx)}+h_1(x)+h_2(x).
\end{equation}
To solve this problem via \eqref{e:long17}, we require the
resolvents $J_{\gamma_n{A}}=\prox_{\gamma_n h_1}$, 
$J_{\mu_{k,n}B_k}=\prox_{\mu_{k,n} g_k}$, and
$J_{\nu_{k,n}D_k}=\prox_{\nu_{k,n} \ell_k}$.
\end{example}
\begin{proof}
First, in view of \cite[Corollaire~10]{Bail77}, $C$ is cocoercive.
Next, for every $k\in K$, it follows from Lemma~\ref{l:4} that
$\partial g_k\infconv\partial\ell_k=\partial(g_k\infconv\ell_k)$.
Since $\overline{x}\in\XX$ solves Problem~\ref{prob:11}, 
\cite[Theorem~2.4.2(vii)--(viii)]{Zali02}
yields
\begin{align}
0&\in \sum_{i\in I}\partial (\phi_i\circ f_i)(\overline{x})
+\sum_{k\in K}\big(L^*_k\circ(\partial g_k\infconv \partial
\ell_k)\circ 
L_k\big)\overline{x}+\partial h_1(\overline{x})+\nabla 
h_2(\overline{x})\nonumber\\
&\subset \partial\bigg(\sum_{i\in I}(\phi_i\circ f_i)
+\sum_{k\in K}( g_k\infconv  \ell_k) 
\circ L_k+ h_1+h_2\bigg)(\overline{x}),
\end{align}
and the result follows from Lemma~\ref{l:2}\ref{l:2iv}.
\end{proof}

\begin{remark}
\label{r:13}
In Example~\ref{ex:32}, if $I$ and $K$ are singletons and $h_2=0$, 
we recover Problem~\ref{prob:1}, which can therefore be solved as
an instance of \eqref{e:long17} in Theorem~\ref{t:3}.
\end{remark}

The next example was suggested to us by Roberto Cominetti.

\begin{example}
\label{ex:13}
Let $\mathsf{X}$ be a reflexive real Banach space, let $J$ and $K$
be disjoint finite sets, and let
$\mathsf{h}_0\in\Gamma_0(\mathsf{X})$. 
For every $j\in J$, let $0<m_j\in\NN$,
let $\Phi_j\in\Gamma_0(\RR^{m_j})$ be
increasing in each component, let 
$(\mathsf{f}_{j,s})_{1\leq s\leq m_j}$ 
be functions in $\Gamma_0(\mathsf{X})$, and let
$\mathsf{F}_j\colon\mathsf{x}\mapsto(\mathsf{f}_{j,1}
(\mathsf{x}),\ldots,\mathsf{f}_{j,m_j}(\mathsf{x}))$. Further, for
every $k\in K$, let $\YY_k$ be a reflexive real Banach space, let
$g_k$ and $\ell_k$ be functions in $\Gamma_0(\YY_k)$ such that
$0\in\sri(\dom g_k^*-\dom\ell_k^*)$, and let
$\mathsf{L}_k\colon\mathsf{X}\to\YY_k$ be linear and bounded.
Consider the problem
\begin{equation}
\label{e:p1vec}
\minimize{\mathsf{x}\in\mathsf{X}}{\sum_{j\in J}
(\Phi_j\circ \mathsf{F}_j)(\mathsf{x})+
\sum_{k\in K}(g_k\infconv\ell_k)(\mathsf{L}_k\mathsf{x})
+\mathsf{h}_0(\mathsf{x})}.
\end{equation}
To show that \eqref{e:p1vec} is a special case of \eqref{e:p3} in 
Example~\ref{ex:32}, let us first observe that, since 
$(\Phi_j)_{j\in J}$ are increasing componentwise, 
\begin{align}
(\forall j\in J)\big(\forall\mathsf{x}\in 
\mathsf{F}_j^{-1}(\dom\Phi_j)\big)\quad 
(\Phi_j\circ\mathsf{F}_j)(\mathsf{x})
&=\min_{\substack{(\zeta_{s})_{1\leq s\leq m_j}\in\RR^{m_j}\\
(\forall s\in\{1,\ldots,m_j\})\;\;
\mathsf{f}_{j,s}(\mathsf{x})\leq \zeta_s}}
\Phi_j(\zeta_1,\ldots,\zeta_{m_j})\nonumber\\
&=\min_{\mathsf{z}\in\RR^{m_j}}\big(
\iota_{\RM^{m_j}}(\mathsf{F}_j(\mathsf{x})-\mathsf{z})
+\Phi_j(\mathsf{z})\big). 
\end{align}
Therefore, \eqref{e:p1vec} is equivalent to 
\begin{equation}
\label{e:p1vecm}
\minimize{\substack{\mathsf{x}\in \mathsf{X}\\(\forall j\in J)
\:\mathsf{z}_j\in \RR^{m_j}}}{\sum_{j\in J}
\iota_{\RM^{m_j}}(\mathsf{F}_j(\mathsf{x})-\mathsf{z}_j)+
\sum_{j\in J}\Phi_j(\mathsf{z}_j)+
\sum_{k\in K}(g_k\infconv\ell_k)(\mathsf{L}_k\mathsf{x})
+\mathsf{h}_0(\mathsf{x})}.
\end{equation}
Next, let us set $\XX=\mathsf{X}\times\Cart_{\!j\in J}\RR^{m_j}$ 
and $I=\menge{(j,s)}{j\in J,\,1\leq s\leq m_j}$. In addition, for
every $j\in J$, let $\mathsf{z}_j=(\zeta_{j,s})_{1\leq s\leq m_j}
\in\RR^{m_j}$ and, for every $i=(j,s)\in I$, let
$\phi_{i}=\iota_{\RM}$ and 
$f_i\colon (\mathsf{x},(\mathsf{z}_j)_{j\in J})
\mapsto \mathsf{f}_{j,s}(\mathsf{x})-\zeta_{j,s}$. 
Finally, set $h_1\colon(\mathsf{x},(\mathsf{z}_j)_{j\in J})\mapsto 
\mathsf{h}_0(\mathsf{x})+\sum_{j\in J}\Phi_j(\mathsf{z}_j)$, 
$h_2=0$, and, for every
$k\in K$, $L_k\colon(\mathsf{x},(\mathsf{z}_j)_{j\in J})\mapsto 
\mathsf{L}_k\mathsf{x}$. Since the functions $(f_i)_{i\in I}$ 
and $h_1$ are in $\Gamma_0(\XX)$, 
\eqref{e:p1vecm} is a particular instance of \eqref{e:p3}.
In connection with solving
\eqref{e:p1vec} via Theorem~\ref{t:3} when $\mathsf{X}$ is a
Hilbert space, let us note that, for every 
$\mu\in\RPP$ and every $i=(j,s)\in I$, the computation of 
$\mathsf{w}_i$ is similar to that in
Example~\ref{ex:1b}, while
\begin{equation}
\prox_{\mu f_{i}}\colon
\big(\mathsf{x},(\mathsf{z}_{j'})_{j'\in J}\big)
\mapsto \big(\prox_{\mu\mathsf{f}_{j,s}}\mathsf{x}\,,
(\mathsf{z}_{j'}+\mu\delta_{j,j'}\mathsf{e}_{j,s})_{j'\in J}\big),
\end{equation}
where $\mathsf{e}_{j,s}$ is the $s$th canonical vector of 
$\RR^{m_j}$ and $\delta$ is the Kronecker symbol. Furthermore, for
every $\mu\in\RPP$, 
\begin{equation}
\label{e:un4}
\prox_{\mu h_1}\colon 
\big(\mathsf{x},(\mathsf{z}_{j})_{j\in J}\big)\mapsto
\big(\prox_{\mu\mathsf{h}_0}\mathsf{x}\,,
(\prox_{\mu\Phi_j}\mathsf{z}_{j})_{j\in J}\big).
\end{equation}
\end{example}

\begin{example}
\label{ex:90}
In Example~\ref{ex:13}, suppose that
\begin{equation}
\label{e:k76}
(\forall j\in J)\quad
\Phi_j\colon(\zeta_{j,s})_{1\leq s\leq m_j}\mapsto
\max_{1\leq s\leq m_j}\zeta_{j,s}. 
\end{equation}
Then \eqref{e:p1vec} reduces to
\begin{equation}
\label{e:14}
\minimize{\mathsf{x}\in \mathsf{X}}{\sum_{j\in J}
\max_{1\leq s\leq m_j}\mathsf{f}_{j,s}(\mathsf{x})+
\sum_{k\in K}(g_k\infconv\ell_k)(\mathsf{L}_k\mathsf{x})
+\mathsf{h}_0(\mathsf{x})}.
\end{equation}
In a Hilbertian setting, using the numerical approach outlined
in Theorem~\ref{t:3}, we can solve \eqref{e:14} via an algorithm
requiring only the proximity operators of the functions
$\mathsf{f}_{j,s}$, $g_k$, $\ell_k$, and $\mathsf{h}_0$. This
appears to be the first proximal algorithm to solve
problems with such a mix of functions involving maxima. As seen in
\eqref{e:un4}, to implement \eqref{e:long17}, we require the
proximity operators of the functions in \eqref{e:k76}. These can be
computed by using \cite[Example~24.25]{Livre1}.
\end{example}

\end{document}